\documentclass[12pt]{amsart}

\usepackage{amsfonts}
\usepackage{amsmath}
\usepackage{amssymb}
\usepackage{amsthm}
\usepackage{mathtools}
\usepackage{hyperref}
\usepackage[capitalize,nameinlink]{cleveref}
\usepackage{microtype}
\usepackage{diffcoeff}

\usepackage[margin=1.2in]{geometry}
\usepackage{mathrsfs}
\usepackage{graphicx}

\usepackage{xcolor}

\usepackage[shortlabels]{enumitem}
\SetEnumerateShortLabel{i}{\((\roman*)\)}
\SetEnumerateShortLabel{a}{\((\alph*)\)}
\newcounter{labeleditem}
\makeatletter
\newcommand\labeleditem[1][]{\item[#1]\refstepcounter{labeleditem}\def\@currentlabel{#1}}
\makeatother

\theoremstyle{plain}
\newtheorem{theorem}{Theorem}[section]
\newtheorem{lemma}[theorem]{Lemma}
\newtheorem{proposition}[theorem]{Proposition}

\theoremstyle{definition}

\newtheorem{remark}[theorem]{Remark}

\newtheorem{example}[theorem]{Example}

\numberwithin{equation}{section}

\newcommand\N{\mathbb{N}}
\newcommand\Z{\mathbb{Z}}
\newcommand\R{\mathbb{R}}
\newcommand\C{\mathbb{C}}

\newcommand\E{\mathcal{E}}
\newcommand\csn{\operatorname{csn}}
\newcommand{\LieGroup}{G}
\newcommand{\Vecs}{E}
\newcommand{\Dim}{n}
\newcommand{\Rep}{\pi}
\newcommand{\Molf}{\chi}

\DeclarePairedDelimiter{\abs}{\lvert}{\rvert}
\DeclarePairedDelimiter{\norm}{\lVert}{\rVert}
\DeclareMathOperator{\Supp}{supp}
\NewDocumentCommand{\Eval}{m m}{\left<#1,#2\right>}
\NewDocumentCommand{\Lloc}{}{L^1_{\operatorname{loc}}}
\NewDocumentCommand{\LieG}{}{G}
\NewDocumentCommand{\LieA}{O{g}}{\mathfrak{#1}}
\NewDocumentCommand{\Cdot}{}{\, \cdot \,} 

\NewDocumentCommand{\RMul}{m}{R_{#1}}
\NewDocumentCommand{\Orbit}{m}{\gamma_{#1}}
\NewDocumentCommand{\Frame}{s o}{%
  \IfBooleanTF{#1}%
  {\IfValueTF{#2}{Y_{#2}}{\mathbf{Y}}}%
  {\IfValueTF{#2}{X_{#2}}{\mathbf{X}}}%
}
\NewDocumentCommand{\Origin}{}{e}
\NewDocumentCommand{\VectorField}{s}{\IfBooleanTF{#1}{Y}{X}}
\NewDocumentCommand{\RRep}{}{\Rep_{R}}

\NewDocumentCommand{\CCont}{}{C_{c}}
\NewDocumentCommand{\Smooth}{s O{\infty}}{\IfBooleanTF{#1}{\mathcal{E}}{C^{#2}}} 
\NewDocumentCommand{\CSmooth}{s O{\infty}}{\IfBooleanTF{#1}{\mathcal{D}}{C_c^{#2}}}    
\NewDocumentCommand{\SmoothV}{O{\infty}}{\Vecs^{#1}}

\begin{document}

\title[Spaces of smooth and ultradifferentiable vectors]{Quasinormability and property \((\Omega)\) for spaces of smooth and ultradifferentiable vectors associated with Lie group representations}

\author[A. Debrouwere]{Andreas Debrouwere}
\address{A. Debrouwere, Department of Mathematics and Data Science \\ Vrije Universiteit Brussel, Belgium\\ Pleinlaan 2 \\ 1050 Brussels \\ Belgium}
\email{Andreas.Debrouwere@vub.be}

\author[M. Huttener]{Michiel Huttener}
\address{M. Huttener, Department of Mathematics: Analysis, Logic and Discrete Mathematics\\ Ghent University\\ Krijgslaan 281\\ 9000 Ghent\\ Belgium}
\email{Michiel.Huttener@UGent.be}

\author[J. Vindas]{Jasson Vindas}
\address{J. Vindas, Department of Mathematics: Analysis, Logic and Discrete Mathematics\\ Ghent University\\ Krijgslaan 281\\ 9000 Ghent\\ Belgium}
\email{Jasson.Vindas@UGent.be}

\thanks {M. Huttener and J. Vindas were supported by the Research Foundation-Flanders through the FWO grant number G067621N}


\begin{abstract}
    We prove that the spaces of smooth and ultradifferentiable vectors associated with a representation of a real Lie group on a Fr\'{e}chet space \(E\) are quasinormable if \(E\) is so.
    A similar result is shown to hold for the linear topological invariant \((\Omega)\).
    In the ultradifferentiable case, our results particularly apply to spaces of Gevrey vectors of Beurling type.
    As an application, we study the quasinormability and the property \((\Omega)\) for a broad class of Fr\'{e}chet spaces of smooth and ultradifferentiable functions on Lie groups globally defined via families of weight functions.
\end{abstract}

\maketitle

\section{Introduction}

In this article we study the quasinormability and the property \((\Omega)\) for spaces of smooth and ultradifferentiable vectors associated with representations of real Lie groups.
In particular, we will provide criteria to determine when a Lie group invariant locally convex space of smooth or ultradifferentiable functions possesses one of these linear topological properties.
Our considerations shall cover the important instance of spaces of Gevrey vectors of Beurling type, which, together with their Roumieu variants, were introduced and thoroughly investigated by Goodman and Wallach \cite{Goodman1,Goodman2,GW}.

The notion of quasinormability for locally convex spaces is due to Grothendieck \cite{Grothendieck}.
The related property \((\Omega)\) for Fr\'{e}chet spaces
goes back to Vogt and Wagner \cite{M-V,V-W} and may be seen as a quantified version of quasinormability within the class of Fr\'{e}chet spaces.
Every Fr\'{e}chet space that satisfies \((\Omega)\) is hence quasinormable.
Both concepts express approximation properties with respect to families of continuous seminorms.
In this regard, our work here is closely connected with classical results of Gårding \cite{Garding47,Garding}, Nelson \cite{Nelson}, and Goodman \cite[Section 3]{Goodman1} about approximations by smooth, analytic, or Gevrey vectors, respectively.

The interest in these two linear topological properties stems, among other things, from the fact they are often crucial hypotheses for the application of various abstract functional analytic tools.
For example, given a surjective continuous linear map \(f\colon X \to Y\) between two Fr\'{e}chet spaces, the map \(f\) lifts bounded sets if \(\ker f\) is quasinormable \cite{B-D,M-V}, while \(f\) admits a continuous linear right inverse if \(\ker f\) satisfies \((\Omega)\) (assuming that \(X\) or \(Y\) is nuclear and \(Y\) satisfies the so-called property \((\operatorname{DN})\)) \cite{M-V, Vogt-87}.
Furthermore, strong duals of quasinormable Fr\'{e}chet spaces are ultrabornological and thus barrelled, whence the Banach-Steinhaus theorem and the open mapping and closed graph theorems of De Wilde may be applied to them.
We also mention that \((\Omega)\) plays an important role in the isomorphic classification theory for Fr\'{e}chet spaces; in fact, \((\Omega)\) is one of the key assumptions to verify if one wants to obtain sequence space representations of function spaces via the structure theory of Fr\'{e}chet spaces (cf.\ \cite{Debrouwere2, Langenbruch2012}).

Quasinormability and \((\Omega)\) have been studied for a variety of concrete Fr\'{e}chet function spaces; see \cite{B-E,B-M} for spaces of continuous functions, \cite{Debrouwere,Langenbruch2016,Wolf} for spaces of analytic functions, and \cite{D-K, VogtPDFG} for smooth kernels of partial differential operators.
One of the goals of this work is to provide a systematic method for establishing both properties for function spaces that are invariant under a Lie group action, which we shall achieve here viewing such spaces as spaces of smooth and ultradifferentiable vectors of Lie group representations.
We remark that the quasinormability of spaces of smooth vectors associated with representations of \((\R^{n},+)\) was studied by the first author in \cite{Debrouwere}.

We now discuss the content of this article in some more detail.
By a representation of a (real) Lie group \(G\) on a Fr\'{e}chet space \(E\) we simply mean a group homomorphism \(\pi\colon G \to \operatorname{GL}(E)\), where \(\operatorname{GL}(E)\) stands for the group of topological isomorphisms of \(E\) (in general, we will not require \(\pi\) to be (strongly) continuous throughout the article).
We refer to the preparatory Sections \ref{sect-Rn}-\ref{Section smooth and ultradifferentiable vectors} for more information about representations and the associated spaces of smooth and ultradifferentiable vectors.
In the smooth case, our main results may now be summarized as follows:
\begin{theorem} \label{main-intro-abstract}
    Let \(\pi\) be a (strongly) continuous
    representation of a Lie group on a Fr\'{e}chet space \(E\).
    Let \(E^\infty\) be the space of smooth vectors associated with \(\pi\).
    \begin{enumerate}[i]
        \item \label{main-intro-abstract-1} \(E^\infty\) is quasinormable if \(E\) is so.
        \item \label{main-intro-abstract-2} \(E^\infty\) satisfies \((\Omega)\) if \(E\) does so.
    \end{enumerate}
\end{theorem}

In fact, we show that part \ref{main-intro-abstract-1} of  \cref{main-intro-abstract} holds not only for Fr\'{e}chet spaces but also for general sequentially complete locally convex Hausdorff spaces if in addition the representation is locally equicontinuous.
\Cref{main-intro-abstract} and its analogue for spaces of ultradifferentiable vectors are shown in \cref{sect-main}.
The proofs of \cref{main-intro-abstract} and its ultradifferentiable counterpart for quasinormability make use of a standard approximation procedure involving approximate identities \cite{Garding, Goodman1}.
This procedure is revisited in Subsection \ref{subs-approx}.
Our analysis of the property \((\Omega)\) for ultradifferentiable vectors however requires some new approximation tools.
Our arguments are then based on the so-called parametrix method, a powerful technique that goes back to Schwartz \cite{Schwartz} and was further developed by Komatsu within the theory of ultradifferentiable functions and ultradistributions \cite{Komatsu}.
In Subsection \ref{subs-parametrix} we present an extension of the parametrix method to the setting of ultradifferentiable vectors associated with Lie group representations by adapting a key idea from the proof of the celebrated Dixmier-Malliavin factorization theorem \cite{D-M}.

It turns out that many Fr\'{e}chet function spaces on a Lie group \(G\) may be identified with spaces of smooth and ultradifferentiable vectors associated with the left- or right-regular representation of \(G\) on a suitably chosen Fr\'{e}chet function space \(E\) (usually \(E\) is a weighted space of continuous or integrable functions).
This makes \cref{main-intro-abstract} and its ultradifferentiable analogue into powerful devices to study the quasinormability and the property \((\Omega)\) for concrete Fr\'{e}chet function spaces.
We carry out this idea in a very general framework in \cref{sect-examplesI,sect-examplesII}.
We end the introduction by discussing an important instance of these ideas.

Let \(G\) be a Lie group and let \(\mathcal{V} = (v_j)_{j \in \N}\) be a pointwise non-decreasing sequence of strictly positive continuous functions on \(G\) satisfying the mild regularity condition
\begin{equation}
    \label{wfs}
    \forall K \subseteq G \text{ compact } \, \forall i \in \N \, \exists j \geq i, C >0 \, \forall x \in G, y \in K \colon v_i(xy) \leq C v_j(x).
\end{equation}
For \(p \in [1,\infty]\), we define
\[
    \mathcal{D}_{L_{\mathcal{V}}^{p}}(G)
    = \{ f \in C^{\infty}(G) \mid v_jDf \in L^p(G), \, \forall D \in U(\mathfrak{g}), j \in \N \}.
\]
Here, the Lebesgue space \(L^p(G)\) is defined with respect to a fixed right-invariant Haar measure on \(G\), while the elements of
the universal enveloping algebra \(U(\mathfrak{g})\) are to be interpreted as left-invariant differential operators.
We endow  \(\mathcal{D}_{L_{\mathcal{V}}^{p}}(G)\) with its natural locally convex topology, for which it becomes a Fr\'{e}chet space.
If \(G = (\R^n, +)\) these spaces are weighted variants of the classical Schwartz spaces \(\mathcal{D}_{L^{p}}(\R^n)\) \cite{Schwartz} (cf. \cite{D-P-V}).
The sequence space representation \(\mathcal{D}_{L^{p}}(\R^n) \cong s \widehat\otimes \ell_{p}\) \cite{Vogt-83} implies that \(\mathcal{D}_{L^{p}}(\R^n)\) satisfies \((\Omega)\) and thus is quasinormable.
Our results from \cref{sect-examplesII} yield the following characterization of the quasinormability and the property \((\Omega)\) for the spaces \(\mathcal{D}_{L_{\mathcal{V}}^{p}}(G)\) in terms of \(\mathcal{V}\):

\begin{theorem} \label{main-intro-applied}
    Let \(\mathcal{V} = (v_j)_{j \in \N}\) be a pointwise non-decreasing sequence of strictly positive continuous functions on a Lie group \(G\) satisfying \eqref{wfs} and let \(p \in [1,\infty]\).
    \begin{enumerate}[a]
        \item Consider the following statements:
            \begin{enumerate}[i]
                \item \label{main-intro-applied-a-1}\(\mathcal{V}\) satisfies the condition
                    \[
                        \forall i \in \N \, \exists j \geq i \, \forall m \geq j \, \forall \varepsilon \in (0,1] \, \exists C >0 \, \forall x \in G \colon 
                        \frac{1}{v_j(x)} \leq \frac{\varepsilon }{v_i(x)} + \frac{C}{v_m(x)}.
                    \]
                \item \label{main-intro-applied-a-2} \(\mathcal{D}_{L^p_{\mathcal{V}}}(G)\)
                    is quasinormable.
            \end{enumerate}
            Then, \ref{main-intro-applied-a-1} \(\Rightarrow\) \ref{main-intro-applied-a-2}.
            If \(p = \infty\) or \(G\) is unimodular, then also \ref{main-intro-applied-a-2} \(\Rightarrow\) \ref{main-intro-applied-a-1}.
        \item Consider the following statements:
            \begin{enumerate}[i]
                \item \label{main-intro-applied-b-1} \(\mathcal{V}\) satisfies the condition
                    \[
                        \forall i \in \N \, \exists j \geq i \, \forall m \geq j \, \exists C,s >0 \, \forall \varepsilon \in (0,1] \, \forall x \in G \colon 
                        \frac{1}{v_j(x)} \leq \frac{\varepsilon}{v_i(x)} + \frac{C}{\varepsilon^s}\frac{1}{v_m(x)}.
                    \]
                \item \label{main-intro-applied-b-2} \(\mathcal{D}_{L^p_{\mathcal{V}}}(G)\)
                    satisfies \((\Omega)\).
            \end{enumerate}
            Then, \ref{main-intro-applied-b-1} \(\Rightarrow\) \ref{main-intro-applied-b-2}.
            If \(p = \infty\) or \(G\) is unimodular, then also \ref{main-intro-applied-b-2} \(\Rightarrow\) \ref{main-intro-applied-b-1}.
    \end{enumerate}
\end{theorem}

\section{Spaces of vector-valued ultradifferentiable functions on open subsets of \texorpdfstring{\(\R^n\)}{Rⁿ}}\label{sect-Rn}
In this short section we recall some facts about vector-valued ultradifferentiable functions on subsets of \(\mathbb{R}^{n}\).
By a \emph{weight function} \cite{BMT}  we mean a continuous increasing function \(\omega\colon [0,\infty) \rightarrow [0,\infty)\) with \(\omega_{|[0,1]} \equiv 0\) satisfying the following properties:
\begin{enumerate}
    \labeleditem[\((\alpha)\)] \label{cnd:alpha} \(\omega(2t) = O(\omega(t))\).
    \labeleditem[\((\beta)\)] \label{cnd:beta} \(\displaystyle \int_{1}^{\infty}\frac{\omega(t)}{t^2} dt < \infty\).
    \labeleditem[\((\gamma)\)] \label{cnd:gamma} \(\log t = o(\omega(t))\).
    \labeleditem[\((\delta)\)] \label{cnd:delta} \(\varphi\colon [0, \infty) \rightarrow [0, \infty)\), \(\varphi(t) = \omega(e^{t})\), is convex.
\end{enumerate}
Since \(\omega\) is increasing, condition \ref{cnd:beta} implies that \(\omega(t) = o(t)\).
\begin{example}\label{exa:gevrey-weight}
    The \emph{Gevrey weight of order} \(s >0\) is defined as
    \[
        \omega_s(t) 
        = \max \{0, t^s - 1 \}
        .
    \]
    We shall always assume that \(s<1\), which ensures that \(\omega_s\) is a weight function.
\end{example}
Throughout the rest of this article we fix a weight function \(\omega\) and write \(\varphi(t) = \omega(e^t)\) (cf.\ condition \ref{cnd:delta} above).
We define
\[
    \varphi^{*} \colon [0, \infty) \rightarrow [0, \infty), \, \varphi^{*}(t) = \sup_{u \geq 0} \{tu- \varphi(u)\}.
\]
The function \(\varphi^{*}\) is increasing,
convex, \(\varphi^*(0) = 0\), \((\varphi^*)^* = \varphi\), and \(\varphi^*(t)/t \nearrow \infty\) on \([0,\infty)\).
We have \cite[Lemma 2.6]{Heinrich}
\begin{equation}
    \label{M12}
    \forall C_1,C_2,h > 0 \, \exists C,k >0 \, \forall t \geq 0 \colon 
    \frac{1}{k}\varphi^*(k(t+C_1)) + C_2t \leq \frac{1}{h}\varphi^*(ht) + \log C.
\end{equation}
The conditions \(\omega(t) = o(t)\) and \eqref{M12} imply that
\begin{equation}
    \label{NA} 
    \forall h,k > 0 \, \exists C >0 \, \forall j \in \N \colon 
    j! \leq Ck^j \exp\left({\frac{1}{h}\varphi^*(hj)}\right).
\end{equation}
\begin{example}\label{gevrey1}
    For the Gevrey weights \(\omega_s\),
    we set \(\varphi_s(t) = \omega_s(e^t)\).
    Then,
    \[
        \exp(\varphi^*_s(t)) 
        = e \left(\frac{1}{se}\right)^{\frac{t}{s}} t^{\frac{t}{s}}.
    \]
    Consequently, we have, for all \(h>0\),
    \[
        \exp \left (\frac{1}{h}\varphi_s^*(ht) \right) 
        = e^{\frac{1}{h}}\left(\frac{h}{se}\right)^{\frac{t}{s}} t^{\frac{t}{s}}.
    \]
\end{example}

\smallskip

Let \(\Theta \subseteq \R^n\) be open and let \(E\) be a lcHs (= locally convex Hausdorff space).
We denote by \(\csn(E)\) the family of all continuous seminorms on \(E\).
Given \(h >0\) we define \(\E^{\omega,h}(\Theta;E)\) as the space consisting of all \(f \in C^\infty(\Theta;E)\)
such that, for all \(K \subseteq \Theta\) compact and \(p \in \csn(E)\),
\[
    p_{K,\omega,h}(f) 
    = \sup_{x \in K} \sup_{\alpha \in \N^n} p(f^{(\alpha)}(x))\exp \left(-\frac{1}{h}\varphi^*(h\abs{\alpha})\right)  
    < \infty
    .
\]
We endow \(\E^{\omega,h}(\Theta;E)\) with the Hausdorff locally convex topology generated by the system of seminorms \(\{ p_{K,\omega,h} \mid K \subseteq \Theta \text{ compact}, p \in \csn(E) \}\).
We set
\[
    \E^{(\omega)}(\Theta;E) 
    = \varprojlim_{h \to 0^+} \E^{\omega,h}(\Theta;E)
    .
\]
We write \(\E^{(\omega)}(\Theta) = \E^{(\omega)}(\Theta;\C)\).
The non-quasianalyticity condition \ref{cnd:beta} means that \(\E^{(\omega)}(\Theta)\) contains non-zero compactly supported functions (see \cite{BMT} for more information).
Let \(\mathcal{A}(\Theta)\) be the space of real analytic functions on \(\Theta\).
The inequality \eqref{NA} implies that \(\mathcal{A}(\Theta) \subseteq \E^{(\omega)}(\Theta)\).
\begin{lemma}\label{comp-ra}
    Let \(\Theta, \Theta' \subseteq \R^n\) be open, and let \(E\) be a lcHs.
    Let \(\phi\colon \Theta \to \Theta'\) be real analytic.
    Then, \(f \circ \phi \in \E^{(\omega)}(\Theta;E)\) for all \(f \in \E^{(\omega)}(\Theta';E)\).
\end{lemma}
\begin{proof} 
    This can be shown by
    adapting the proof of \cite[Proposition 8.4.1]{Hormander}; the details are left to the reader.
\end{proof}

\section{Spaces of vector-valued smooth and ultradifferentiable functions on manifolds}\label{sect-Man}

We shall now discuss vector-valued smooth and ultradifferentiable functions on manifolds.
In this section we fix a smooth manifold \(M\) of dimension \(n\).
In the ultradifferentiable case, we shall always tacitly assume the manifold to be real analytic.
Throughout this article the term \emph{regular} will mean smooth if the manifold under consideration is smooth and real analytic if it is real analytic.

Let \(E\) be a lcHs and \(j \in \N \cup \{\infty\}\).
We define \(C^j(M,E)\) (\(\E^{(\omega)}(M;E)\)) as the space of all those \(f\colon M \to E\) such that \(f \circ \phi^{-1} \in C^j(\phi(U), E)\) (\(f \circ \phi^{-1} \in \E^{(\omega)}(\phi(U), E)\)) for all regular charts \((\phi,U)\) of \(M\).
We endow \(C^j(M;E)\) (\(\E^{(\omega)}(M;E)\)) with the initial topology with respect to the mappings
\begin{gather*}
    C^j(M;E) \to C^j(\phi(U), E), \, f \mapsto f \circ \phi^{-1} \\
    (\E^{(\omega)}(M;E) \to \E^{(\omega)}(\phi(U), E), \, f \mapsto f \circ \phi^{-1})
    ,
\end{gather*}
where \((\phi,U)\) runs over all regular charts of \(M\).
The spaces \(C^j(M;E)\) and \(\E^{(\omega)}(M;E)\) are sequentially complete if \(E\) is so.
We write \(C(M;E) = C^0(M;E)\), \(C^j(M) = C^j(M;\C)\), and \(\E^{(\omega)}(M) = \E^{(\omega)}(M;\C)\).
Note that Lemma \ref{comp-ra} guarantees that we could have defined \(\E^{(\omega)}(M;E)\)  simply using a given regular atlas, and that then its definition does not depend on the choice of the particular atlas.
The latter is obviously true as well for \(C^j(M;E)\).

We denote by \(C_c^j(M)\), \(j \in \N \cup \{\infty\}\), the subspace of \(C^j(M)\) consisting of elements with compact support.
We write \(C_c(M) = C^0_c(M)\) and \(\mathcal{D}(M) = C^\infty_c(M)\).
We set \(\mathcal{D}^{(\omega)}(M) = \mathcal{E}^{(\omega)}(M) \cap C_c(M)\).
The non-quasianalyticity assumption \ref{cnd:beta} again ensures the non-triviality of \(\mathcal{D}^{(\omega)}(M)\).
We endow the spaces \(C_c^j(M)\) and \(\mathcal{D}^{(\omega)}(M)\) with their natural \((LF)\)-space topology.
For \(K \subseteq M\) compact we denote by \(\mathcal{D}_K\) (\(\mathcal{D}^{(\omega)}_K\)) the subspace of \(C^\infty(M)\) (\(\E^{(\omega)}(M)\)) consisting of elements with support in \(K\).

Let \(E\) be a lcHs and let \(f \colon M \to E\).
For \(v' \in E'\) we define the mapping
\[
    \Eval{v'}{f}\colon M \to \C, \, x \mapsto \Eval{v'}{f(x)}
    .
\]
If \(f \in C^\infty(M;E)\), then \(\Eval{v'}{f} \in C^\infty(M)\) and it holds that \(X \Eval{v'}{f} = \Eval{v'}{X f}\)  for each vector field \(X\) on \(M\).

\begin{lemma}\label{weaksmooth}
    Let \(E\) be a sequentially complete lcHs and let \(f \colon M \to E\).
    Then, \(f \in C^\infty(M;E)\) if and only if \(\Eval{v'}{f} \in C^\infty(M)\) for all \(v' \in E'\).
\end{lemma}
\begin{proof}
    The result is well-known if \(M\) is an open subset of \(\R^n\) (cf.\ \cite[Appendice Lemme II]{Schwartz-vectorvaluedI}).
    The general case follows from it by using local coordinates.
\end{proof}
Let \(U \subseteq M\) be open.
A system \(\mathbf{X} = (X_1, \ldots, X_n)\) of vector fields on \(U\) is called a \emph{frame} on \(U\) if \((\left. X_1 \right|_x, \ldots, \left. X_n \right|_x)\) is a basis of \(T_xU\) at each point \(x \in U\).
The frame \(\mathbf{X}\) is said to be regular if the vector fields \(X_1, \ldots, X_n\) are regular on \(U\).
 We write \(X^\alpha = X_{\alpha_1} \cdots X_{\alpha_j}\) for \(\alpha = (\alpha_1, \ldots, \alpha_j) \in \{1, \ldots, n\}^j\), \(j \in \N\).

Let \(\mathbf{X}\) be a regular frame on the open subset \(U\subseteq M\) and let \(E\) be a lcHs.
Let \(K\) be a compact subset of \(U\)
and \(p \in \csn(E)\).
For \(j \in \N\), we define
\[
    p_{\mathbf{X},K,j}(f) 
    = \max_{i \leq j}\max_{\alpha \in \{1, \ldots, n\}^i} \sup_{x \in K}p(X^\alpha f(x))
    , \qquad f \in C^\infty(U;E)
    .
\]
Given \(h >0\), we also define
\[
    p_{\mathbf{X}, K,\omega,h}(f) 
    = \sup_{j \in \N} \max_{\alpha \in \{1, \ldots, n\}^j} \sup_{x \in K} p(X^\alpha f(x))\exp \left(-\frac{1}{h}\varphi^*(hj)\right)
    , \qquad f \in C^\infty(U;E)
    .
\]
In the special case \(E = \C\) and \(p = \abs{ \Cdot }\), we write \(p_{\mathbf{X},K,j} = \norm{ \Cdot }_{\mathbf{X}, K,j}\) and \(p_{\mathbf{X}, K,\omega,h} = \norm{ \Cdot }_{\mathbf{X}, K,\omega,h}\).
Given \(p \in \csn(E)\), we set \(V_p = \{ v \in E \mid p(v) \leq 1 \}\) and write \(V^\circ_p\) for its polar set in \(E'\).
The bipolar theorem yields that, for all \(f \in C^\infty(U;E)\),
\begin{equation}\label{scalartovector}
    p_{\mathbf{X}, K,j}(f) = \sup_{v' \in V^\circ_p} \norm{ \Eval{v'}{f} }_{\mathbf{X}, K,j}, \qquad p_{\mathbf{X}, K,\omega,h}(f) = \sup_{v' \in V^\circ_p} \norm{ \Eval{v'}{f} }_{\mathbf{X}, K,\omega,h}.
\end{equation}
The next result will be frequently used throughout this article.
\begin{proposition}\label{switch-explicit} 
    Let \(U \subseteq M\) be open, let \(\mathbf{X}\) and \(\mathbf{Y}\) be regular frames on \(U\), and let \(E\) be a lcHs.
    \begin{enumerate}[i]
        \item \label{switch-explicit-1} 
            For all \(K \subseteq U\) compact and \(j \in \N\) there is \(C > 0\) such that for all \(p \in \csn(E)\)
            \[
        p_{\mathbf{X},K,j}(f) 
                \leq C p_{\mathbf{Y},K,j}(f)
                , \qquad \forall f \in C^j(U;E)
                .
            \]
        \item \label{switch-explicit-2} 
            For all \(K \subseteq U\) compact and \(h >0\) there are \(C,k > 0\) such that for all \(p \in \csn(E)\)
            \[
    p_{\mathbf{X},K,\omega, h}(f)
                \leq C p_{\mathbf{Y},K,\omega,k}(f)
                , \qquad \forall f \in C^\infty(U;E)
                .
            \]
    \end{enumerate}
\end{proposition}
\begin{proof}
    It suffices to consider the case \(E = \C\), the general case follows from it by \eqref{scalartovector}.
    Furthermore, using local coordinates and a compactness argument, we may assume that \(U\) is an open subset of \(\R^n\).
    The statement \ref{switch-explicit-1} is clear.
    We now show \ref{switch-explicit-2}.
    Let \(K\) be an arbitrary compact subset of \(U\). 
    We define, for \(j \in \N\), 
    \[
        \abs{f}_{{\mathbf{X}},K,j} 
        = \max_{\alpha \in \{1, \ldots, n\}^j} \sup_{x\in K} \abs{X^\alpha f(x)}
        , \quad 
        \abs{f}_{{\mathbf{Y}},K,j} 
        = \max_{\alpha \in \{1, \ldots, n\}^j} \sup_{x\in K} \abs{Y^\alpha f(x)}
        , \qquad f \in C^\infty(U)
        .
    \]
    There are \(a_{\ell,m} \in \mathcal{A}(U)\), \(\ell,m = 1, \ldots,n\), such that
    \[
        X_{\ell}(f)(x) 
        = \sum_{m=1}^n a_{\ell,m}(x) Y_m(f)(x)
        , \qquad f \in C^\infty(U)
        ,
    \]
    for all \(\ell = 1, \ldots,n\).
    Since \(a_{\ell,m} \in \mathcal{A}(U)\) there is \(H\geq 1\) such that
    \[
        \abs{a_{\ell,m}}_{{\mathbf{X}},K,j} 
        \leq H^{j+1} j!
        , \qquad j \in \N
        ,
    \]
    for all \(\ell,m = 1, \ldots,n\) \cite[Theorem 3.1]{Goodman1} (which is originally a result due to Nelson \cite{Nelson}).
    We claim that for all \(j\in \N \setminus \{0\}\)
    \begin{equation}
        \label{claim}
        \abs{f}_{{\mathbf{X}},K,j} 
        \leq (2nH)^j \sum_{i=1}^j \binom{j}{i} (j-i)! \abs{f}_{{\mathbf{Y}},K,i}, \qquad  f \in C^\infty(U).
    \end{equation}
    Before we show the claim, let us show how the result follows from it.
    Let \(h >0\) be arbitrary.
    By \eqref{M12}, there are \(C,k >0\) such that
    \begin{equation}
        \label{prop11}
        \frac{1}{k}\varphi^*(kt) + t\log(4nH) \leq \frac{1}{h}\varphi^*(ht) + \log C, \qquad t \geq 0.
    \end{equation}
    The convexity of \(\varphi^*\), \eqref{claim}, and \eqref{prop11} imply that for all for \(j\in \N \setminus \{0\}\) and \(f \in C^\infty(U)\)
    \begin{align*}
        \abs{f}_{{\mathbf{X}},K,j} \exp \left(-\frac{1}{h}\varphi^*(hj)\right) 
        &\leq \frac{1}{2^j} \sum_{i=1}^j \binom{j}{i} (4nH)^{j-i}(j-i)!\exp \left(-\frac{1}{h}\varphi^*(h(j-i))\right) \times \\
        & \qquad (4nH)^{i}\abs{f}_{{\mathbf{Y}},K,i}\exp \left(-\frac{1}{h}\varphi^*(hi)\right) \\
        &\leq C' \norm{f}_{{\mathbf{Y}},K,\omega,k}
        ,
    \end{align*}
    where
    \[
        C' 
        = C\sup_{l \in \N} (4nH)^{l}l!\exp \left(-\frac{1}{h}\varphi^*(hl)\right) 
        < \infty
        ,
    \]
    see \eqref{NA}.
    This shows the result.
    We now prove \eqref{claim} by induction.
    The case \(j =1\) is clear (recall that \(H \geq 1\)).
    Let \(j\in \N \setminus \{0\}\) and assume that \eqref{claim} holds for all \(l = 1, \ldots, j\).
    Let \(f \in C^\infty(U)\) be arbitrary.
    For all \(\alpha \in \{1, \ldots, n\}^j\), \(\ell \in \{1, \ldots, n\}\), and \(x \in K\),
    \begin{align*}
         &\abs{X^\alpha X_\ell f(x)} \leq \sum_{m=1}^n \abs{X^{\alpha}(a_{\ell,m}(x) Y_m(f)(x))} \\
         &\leq \sum_{m=1}^n \sum_{r=0}^j \binom{j}{r} \abs{a_{\ell,m}}_{{\mathbf{X}},K,j-r} \abs{Y_m(f)}_{{\mathbf{X}},K,r} \\
         &\leq n\left(H^{j+1}j! \abs{f}_{{\mathbf{Y}},K,1} + \sum_{r=1}^j \binom{j}{r} H^{j-r+1} (j-r)!  (2nH)^r \sum_{i=1}^r \binom{r}{i} (r-i)! \abs{f}_{{\mathbf{Y}},K,i+1}\right) \\
         &\leq (nH)^{j+1}\left(j! \abs{f}_{{\mathbf{Y}},K,1} + \sum_{r=1}^j\sum_{i=1}^r \frac{j!}{i!} 2^r \abs{f}_{{\mathbf{Y}},K,i+1}\right) \\
         &\leq (2nH)^{j+1}\left(j! \abs{f}_{{\mathbf{Y}},K,1} + \sum_{i=1}^j \binom{j+1}{i+1}(j-i)! \abs{f}_{{\mathbf{Y}},K,i+1}\right) \\
         &\leq (2nH)^{j+1} \sum_{i=1}^{j+1} \binom{j+1}{i}(j+1-i)! \abs{f}_{{\mathbf{Y}},K,i}.
    \end{align*}
    This shows the induction step. 
\end{proof}

Proposition \ref{switch-explicit} yields the following result.

\begin{proposition}\label{main-frames} 
    Let \((U_i)_{i \in I}\) be an open covering of \(M\), let \(\mathbf{X}_i\) be a regular frame on \(U_i\) for each \(i \in I\), and let \(E\) be a lcHs.
    \begin{enumerate}[i]
        \item \label{main-frames-1}
            The locally convex topology of \(C^\infty(M;E)\) is generated by the system of seminorms \(\{ p_{\mathbf{X}_i,K_i,j} \mid i \in I, K_i \subseteq U_i \text{ compact}, j \in \N, p \in \csn(E) \}\).
        \item \label{main-frames-2} 
            \(f \in C^\infty(M;E)\) belongs to \(\E^{(\omega)}(M;E)\) if and only if \(p_{\mathbf{X}_i, K_i,\omega,h}(f) < \infty\) for all \(i \in I\), \(K_i \subseteq U_i\) compact, \(h >0\), and \(p \in \csn(E)\).
            Moreover, the locally convex topology of \(\E^{(\omega)}(M;E)\) is generated by the system of seminorms \(\{ p_{\mathbf{X}_i,K_i,\omega,h} \mid i \in I, K_i \subseteq U_i \text{ compact}, h>0, p \in \csn(E) \}\).
    \end{enumerate}
\end{proposition}

\begin{remark}
    All results from \cref{sect-Rn,sect-Man} remain valid if one replaces the condition \ref{cnd:beta} on \(\omega\) by the weaker assumption
    \(\omega(t) = o(t)\).
\end{remark}

\begin{remark}
    The theorem of iterates from \cite[Corollary 3.19]{FS} also yields the set equality part from Proposition \ref{main-frames}\ref{main-frames-2}, that is, the equivalence of \(f\in \E^{(\omega)}(M;E)\) to having \(p_{\mathbf{X}_i, K_i,\omega,h}(f) < \infty\) for all of these seminorms.
    However, it should be noticed that the method from \cite{FS}, being based on reducing the Beurling case to the Roumieu case (see the proof of \cite[Proposition 3.2]{FS}), does not directly deliver the required continuity estimates needed to conclude that the system \(\{ p_{\mathbf{X}_i,K_i,\omega,h} \mid i \in I, K_i \subseteq U_i \text{ compact}, h>0, p \in \csn(E) \}\) generates the locally convex topology of \(\E^{(\omega)}(M;E)\).
\end{remark}

\section{Spaces of smooth and ultradifferentiable vectors associated with Lie group representations}
\label{Section smooth and ultradifferentiable vectors}
Throughout the rest of this article we fix a (real) Lie group \(G\) of dimension \(n\) with identity element \(e\).
In the ultradifferentiable case, we shall always tacitly view the underlying manifold structure of \(G\) as a real analytic one (cf.\ \cref{sect-Man}).
Given a function \(f\) on \(G\) and \(x \in G\), we define
\[
    \check{f}(y) = f(y^{-1})
    , \qquad 
    L_x f(y) = f(x^{-1}y)
    , \qquad 
    R_x f(h) = f(yx)
    , \qquad 
    y \in G
    .
\]
A frame \({\mathbf{X}} = \{X_1, \ldots, X_n\}\) on \(G\) is said to be left-invariant (right-invariant) if the vector fields \(X_1, \ldots, X_n\) are left-invariant (right-invariant).
Every left-invariant (right-invariant) frame on \(G\) is automatically regular and corresponds to the left-invariant (right-invariant) vector fields associated with a basis in the Lie algebra of \(G\).
We fix a left Haar measure on \(G\).
Unless explicitly stated otherwise, all integrals on \(G\) will be meant with respect to this measure.
All vector-valued integrals in this article are to be interpreted in the weak sense.
If \(E\) is a sequentially complete lcHs, the integral \(\int_G f(x) dx\) exists for all \(f \in C_c(G;E)\).
Let \(E\) be a lcHs.
We denote by \(\operatorname{GL}(E)\) the group of topological isomorphisms of \(E\).
By a \emph{representation} of \(G\) on \(E\) we mean a group homomorphism \(\pi\colon G \to \operatorname{GL}(E)\).
The representation \(\pi\) is called \emph{locally equicontinuous} if \(\{ \pi(x) \mid x \in K\}\) is equicontinuous for each compact set \(K \subseteq G\).
 The mapping
\[
    \gamma_v \colon G \to E, \, x \mapsto \pi(x)v
\]
is called the \emph{orbit} of a given vector \(v\in E\).
We denote by \(E^0\) the space of all those \(v \in E\) such that \(\gamma_v \in C(G,E)\).
The representation is said to be \emph{continuous} if \(E= E^0\).

\begin{remark}
    \label{rmk:continuity-implies-local-equicontinuity}
    Each continuous representation of \(G\) on a barrelled lcHs \(E\) is automatically locally equicontinuous, as follows from the Banach-Steinhaus theorem.
    In particular, this holds if \(E\) is a Fr\'{e}chet space or, more generally, an \((LF)\)-space.
\end{remark}

Let \(\pi\) be a representation of \(G\) on a lcHs \(E\).
Let \(j \in \N \cup \{\infty\}\). 
We define \(E^j\) (\(E^{(\omega)}\)) as the space consisting of all \(v \in E\) such that \(\gamma_v \in C^j(G,E)\) (\(\gamma_v \in \E^{(\omega)}(G,E)\)) and endow this space with the initial topology with respect to the mapping
\begin{gather*}
   E^{j} \to C^j(G,E), \, v \mapsto \gamma_v \\
    (E^{(\omega)} \to \E^{(\omega)}(G,E), \, v \mapsto \gamma_v)
    .
\end{gather*}
Note that \(E^j\) and \(E^{(\omega)}\) are sequentially complete if \(E\) is so.
Let \(\mathbf{X}\) be a regular frame on \(G\).
Let \(K \subseteq G\) be compact and \(p \in \csn(E)\).
We write \(p_{\mathbf{X},K,j}(v) = p_{\mathbf{X},K,j}(\gamma_v)\), \(j \in \N\), and \(p_{\mathbf{X},K,\omega, h}(v) = p_{\mathbf{X},K,\omega, h}(\gamma_v)\), \(h >0\), for \(v \in E^\infty\).
Proposition \ref{main-frames} implies that the locally convex topology of \(E^\infty\) is generated by the system of seminorms \(\{ p_{\mathbf{X},K,j} \mid K \subseteq G \text{ compact}, j \in \N, p \in \csn(E) \}\).
Moreover,
\[
    E^{(\omega)}
    = \{v\in E^\infty \mid p_{\mathbf{X}, K,\omega,h}(v) < \infty \text{ for all } K \subseteq G \text{ compact}, h>0, p \in \csn(E) \}
\]
and the system of seminorms \(\{ p_{\mathbf{X},K,\omega,h} \mid K \subseteq G \text{ compact}, h>0, p \in \csn(E) \}\) generates its locally convex topology.

Let \(\pi\) be again a representation of \(G\) on a lcHs \(E\) and let \({\mathbf{X}}\) be a left-invariant frame on \(G\).
For $j \in \N$ and \(\alpha \in \{1, \ldots, n\}^j\), we define (the induced infinitesimal action of the representation)
\begin{equation}\label{eq:infinitesimalaction}
   \pi(X^{\alpha})v
    = X^\alpha \gamma_v(e)
    , \qquad v \in E^{j}
    .
\end{equation}
Since \({\mathbf{X}}\) is left-invariant, we obtain
\begin{equation}
    \label{orbit-prop}
    \gamma_{{\pi(X^{\alpha})v}} = X^\alpha \gamma_v, \qquad v \in E^{j}.
\end{equation}
Let \(p \in \csn(E)\).
For \(j \in \N\), we define
\[
    p_{{\mathbf{X}},j}(v) 
    = \max_{i \leq j} \max_{\alpha \in \{1, \ldots, n\}^i} p({\pi(X^{\alpha})v})
    , \qquad v \in E^\infty
    ,
\]
and, for \(h >0\), 
\begin{equation}
    \label{eq seminorms ultra vectors}
    p_{{\mathbf{X}},\omega,h}(v) 
    = \sup_{j \in \N} \max_{\alpha \in \{1, \ldots, n\}^j} p({\pi(X^{\alpha})v})\exp \left(-\frac{1}{h}\varphi^*(hj)\right)
    , \qquad v \in E^\infty
    .
\end{equation}
\begin{proposition}\label{main-frames-vec-li} 
    Let \(\pi\) be a locally equicontinuous representation of \(G\) on a lcHs \(E\) and let \(\mathbf{X}\) be a left-invariant frame on \(G\).
    \begin{enumerate}[i]
        \item \label{main-frames-vec-li-1}
            The locally convex topology of \(E^\infty\) is generated by the system of seminorms \(\{ p_{\mathbf{X},j} \mid j \in \N, p \in \csn(E) \}\).
        \item \label{main-frames-vec-li-2} 
            \(v \in E^\infty\) belongs to \(E^{(\omega)}\) if and only if \(p_{\mathbf{X},\omega,h}(v) < \infty\) for all \(h > 0\) and \(p \in \csn(E)\).
            Moreover, the locally convex topology of \(E^{(\omega)}\) is generated by the system of seminorms \(\{ p_{\mathbf{X},\omega,h} \mid h >0, p \in \csn(E) \}\).
    \end{enumerate}
\end{proposition}
\begin{proof}
    In view of \eqref{orbit-prop}, this is a consequence of Proposition \ref{main-frames} and the fact that \(\pi\) is locally equicontinuous.
\end{proof}

\begin{example}
    For the Gevrey weights \(\omega_s\), \cref{gevrey1} implies that the system of seminorms \(\{ p_{\mathbf{X},\omega_s,h} \mid h >0, p \in \csn(E) \}\) is equivalent to \(\{\widetilde{p}_{\mathbf{X},\omega_s,h} \mid h >0, p \in \csn(E) \}\), where
    \[
        \widetilde{p}_{\mathbf{X},\omega_s,h}(v)
        = \sup_{j \in \N} \max_{\alpha \in \{1, \ldots, n\}^j} \frac{p({\pi(X^{\alpha})v})}{h^j j!^{1/s}}
        , \qquad v \in E^\infty
        .
    \]
    Hence, \(\{ \widetilde{p}_{\mathbf{X},\omega_s,h} \mid h >0, p \in \csn(E) \}\) also generates the locally convex topology of \(E^{(\omega)}\).
    These spaces and their Roumieu type variants were considered by Goodman and Wallach \cite{Goodman1,Goodman2,GW} (cf.\ the introduction).
\end{example}

Let \(E\) be a sequentially complete lcHs.
For \(f \in C(G,E)\) and \(\chi \in C_c(G)\), we define their (left-)convolution as
\[
    (f \ast \chi)(x) = \int_{G} f(y) \chi(y^{-1}x) dy
    = \int_{G} f(xy) \chi(y^{-1}) dy
    , \qquad x \in G
    .
\]
Note that
\[
    f \ast \Molf
    = \int_{\LieGroup} f(y) \, L_{y}\Molf dy
    = \int_{\LieGroup} R_{y}f \, \Molf(y^{-1}) dy
    .
\]
Let \(j \in \N \cup \{ \infty \}\).
If \(\chi \in C^j_c(G)\) or \(f \in C^j(G,E)\), then \(f \ast \Molf \in C^j(G,E)\).
Let \(X\) be a vector field on \(G\).
If \(X\) is left-invariant, then
\[
    X(f \ast \Molf) 
    = f \ast (X \Molf )
    , \qquad f \in C(G,E), \chi \in C_c^1(G)
    ,
\]
while if \(X\) is right-invariant, then
\[
    X(f \ast \Molf) 
    = (Xf) \ast \Molf
    , \qquad f \in C^1(G,E), \chi \in C_c(G)
    .
\]
If \(E = \C\), these statements hold under the weaker assumption that \(f \in \Lloc(G)\).

Let \(\pi\) be a representation of \(G\) on a sequentially complete lcHs \(E\).
For \(\chi \in C_c(G)\), we define the operator $\Pi(\chi): E^0\to E^0$ by
\[
    \Pi(\chi)v 
    = \int_{\LieGroup} \gamma_{v}(x) \Molf(x^{-1}) dx, \qquad v\in E^0.
\]
Note that \(\gamma_{\Pi(\chi)v} = \gamma_v \ast \chi\).
Consequently, for \(j \in \N \cup \{ \infty \}\), it holds that \(\Pi(\chi)v \in E^j\) if \(\chi \in C^j_c(G)\).

Let \(\bf{X}\) be a regular frame on \(G\).
For \(j \in \N\) and \(h >0\), we define
\[
    \norm{ f}_{{\bf{X}},j} 
    = \sup_{x \in G} \max_{i \leq j} \max_{\alpha \in \{ 1, \ldots ,n\}^i} \abs{ X^\alpha f(x)}
    , \qquad f \in C^\infty(G),
\]
and
\[
    \norm{ f}_{{\bf{X}},\omega,h} 
    = \sup_{x \in G} \sup_{j \in\N} \max_{\alpha \in \{ 1, \ldots ,n\}^j} \abs{ X^\alpha f(x)}\exp \left(-\frac{1}{h}\varphi^*(hj)\right)
    , \qquad f \in C^{\infty}(G)
    .
\]

\begin{lemma}\label{smoothening}
    Let \(\pi\) be a locally equicontinuous representation of \(G\) on \(E\) and let \(\bf{X}\) be a regular frame on \(G\).
    Let \(K,L \subseteq \LieGroup\) be compact and \(p \in \csn(E)\).
    \begin{enumerate}[i]
        \item \label{smoothening-1}
            There is \(q \in \csn(\Vecs)\) such that for all \(j \in \N\) there is \(C >0\) such that
            \[
                p_{{\bf{X}},K,j}(\Pi(\chi)v) 
                \leq C\norm{ \chi }_{{\bf{X}},j} q(v) 
                , \qquad \forall v \in E^0, \chi \in C^j_c(G) \text{ with } \operatorname{supp} \chi \subseteq L
                .
            \]
        \item \label{smoothening-2} 
            There is \(q \in \csn(\Vecs)\) such that for all \(h >0\) there are \(C,k >0\) such that
            \[
                p_{{\bf{X}},K,\omega,h}(\Pi(\chi)v) 
                \leq C \norm{ \chi }_{{\bf{X}},\omega,k} q(v)
                , \qquad \forall v \in E^0, \chi \in \mathcal{D}_L
                .
            \]
             Moreover, if \({\bf{X}}\) is left-invariant, this holds with \(k = h\).
    \end{enumerate}
\end{lemma}

\begin{proof}
First, suppose \(\bf{X}\) is left-invariant.
Then, for all \(v \in E^0\), \(\chi \in \mathcal{D}(G)\), and \(\alpha \in \{1, \ldots, n\}^j\), \(j \in \N\), it holds that
    \[
        X^\alpha \gamma_{\Pi(\chi)v} 
        = X^\alpha( \gamma_v \ast \chi) 
        = \gamma_{v} \ast (X^\alpha \chi) 
        = \int_G R_y \gamma_v \, X^{\alpha}\Molf(y^{-1}) dy
        .
    \]
    The result now follows from the fact that \(\pi\) is locally equicontinuous.
      
  The general case follows by applying Proposition \ref{switch-explicit}.
\end{proof}

\section{Some auxiliary results}
\label{sect-aux}
\subsection{An approximation result}\label{subs-approx} We fix a regular chart \((\phi,U)\) on \(G\) with \(\phi(e) = 0\) and \(\overline{B}(0,1) \subseteq \phi(U)\).
We define \(S_\varepsilon = \phi^{-1}(\overline{B}(0,\varepsilon))\) for \(\varepsilon \in (0,1]\).

\begin{lemma} \label{MVT}
    Let \({\bf{X}} = (X_1, \ldots, X_n)\) be a smooth frame on \(G\) and let \(E\) be a lcHs.
    For all compact subsets \(K \subseteq G\)  there is \(C >0\) such that for all \(p \in \csn(E)\) and \(\varepsilon \in (0,1]\)
    \begin{equation}
        \label{ineq-1}
        \sup_{x \in K}\sup_{y \in S_\varepsilon} p(f(x) - f(xy)) 
        \leq \varepsilon C \sup_{x \in K} \sup_{y \in S_1} \max_{i \in \{ 1, \ldots, n\}} p( X_i f(xy))
        , \qquad f \in C^\infty(G;E)
        .
\end{equation}
\end{lemma}
\begin{proof}
    It suffices to consider the case \(E = \C\), the general case follows from it by \eqref{scalartovector}.
    Furthermore, we may assume that \({\bf{X}}\) is left-invariant.
    In this case, we will show the following stronger property:
    There is \(C >0\) such that for all \(\varepsilon \in (0,1]\) it holds that
    \[
        \sup_{y \in S_\varepsilon} \abs{f(x) - f(xy)} 
        \leq \varepsilon C \sup_{y \in S_1} \max_{i \in \{ 1, \ldots, n\}} \abs{X_i f(xy)}
        , \qquad f \in C^\infty(G), x \in G
        .
    \]
    Set \(\psi = \phi^{-1}\) and denote by \(\psi_\ast \partial_i\) the pushforward of \(\partial_i\) by \(\psi\), \(i =1, \ldots,n\).
    There is \(C >0\) such that
    \begin{equation}
        \label{invariance}
        \max_{i \in \{ 1, \ldots, n\}} \abs{ (\psi_\ast \partial_i) f(y) }\leq C\max_{i \in \{ 1, \ldots, n\}} \abs{ X_i f(y)}, \qquad f \in C^\infty(G), y \in S_1.
    \end{equation}
    Let \(f \in C^\infty(G)\) and \(x \in G\) be arbitrary.
    The mean value theorem implies that for all \(\varepsilon \in (0,1]\)
    \begin{align*}
        \sup_{y \in S_\varepsilon} \abs{f(x) - f(xy)} 
        &= \sup_{y \in S_\varepsilon} \abs{ (L_{x^{-1}} f)(e) - (L_{x^{-1}} f)(y)} \\
        &= \sup_{\xi \in \overline{B}(0,\varepsilon)} \abs{ ((L_{x^{-1}} f) \circ \psi)(0) - ((L_{x^{-1}} f) \circ \psi)(\xi)} \\
        &\leq \varepsilon \sqrt{n}\sup_{t \in \overline{B}(0,1)} \max_{i \in \{ 1, \ldots, n\}}  \abs{ \partial_i ((L_{x^{-1}} f) \circ \psi)(t)} \\
        &= \varepsilon \sqrt{n}\sup_{y \in S_1 } \max_{i \in \{ 1, \ldots, n\}} \abs{ (\psi_\ast \partial_i)(L_{x^{-1}} f) (y) }.
    \end{align*}
    Applying \eqref{invariance} to \(L_{x^{-1}}f\) and using the fact that \({\bf{X}}\) is left-invariant, we find that
    \[
        \sup_{y \in S_\varepsilon} \abs{f(x) - f(xy)} 
        \leq \varepsilon C\sqrt{n}\sup_{y \in S_1 } \max_{i \in \{ 1, \ldots, n\}} \abs{ X_i(L_{x^{-1}} f) (y) } 
        = \varepsilon C\sqrt{n}\sup_{y \in S_1 } \max_{i \in \{ 1, \ldots, n\}} \abs{ X_if (xy) } 
        .
    \]
\end{proof}

We are ready to show the main result of this subsection.
For the Gevrey weights \(\omega_s\), it may be considered as a quantified version of \cite[Theorem 3.2 and Corollary 3.1]{Goodman1}.

\begin{proposition}\label{approx}
    Let \(\pi\) be a representation of \(G\) on a sequentially complete lcHs \(E\) and let \(\bf{X}\) be a regular frame on \(G\).
    Let \(\ (\Molf_{\varepsilon})_{\varepsilon \in (0,1]} \subseteq \mathcal{D}(G) \) be such that \(\chi_\varepsilon \geq 0\), \(\operatorname{supp} \check \chi_\varepsilon \subseteq S_\varepsilon\), and \(\int_G \check{\chi}_\varepsilon(x) dx = 1\) for all \(\varepsilon \in (0,1]\).
    \begin{enumerate}[i]
        \item \label{approx-1}
            For all \(K \subseteq \LieGroup\) compact and \(j \in \N\) there is \(C>0\) such that for all \(p \in \csn(E)\) and \(\varepsilon \in (0,1]\)
            \[
                p_{{\mathbf{X}},K,j}(v-\Pi(\Molf_{\varepsilon})(v)) 
                \leq \varepsilon  Cp_{{\mathbf{X}},KS_1,j+1}(v)
                , \qquad v \in E^\infty
                .
            \]
        \item \label{approx-2} 
            Suppose that \((\Molf_{\varepsilon})_{\varepsilon \in (0,1]} \subseteq \mathcal{D}^{(\omega)}(G) \).
            For all \(h >0\) there is \(k >0\) such that for all \(K \subseteq \LieGroup\) compact there is \(C >0\) such that for all \(p \in \csn(E)\) and \(\varepsilon \in (0,1]\)
            \[
                p_{{\mathbf{X}},K,\omega,h}(v-\Pi(\Molf_{\varepsilon})(v)) 
                \leq \varepsilon C p_{{\mathbf{X}},KS_1,\omega,k}(v)
                , \qquad v \in E^{(\omega)}
                .
            \]
    \end{enumerate}
\end{proposition}
\begin{proof}
    We will only show \ref{approx-2} as the proof of \ref{approx-1} is similar but simpler.
    In view of Proposition \ref{switch-explicit}\ref{switch-explicit-2}, we may assume that \({\bf{X}}= (X_1, \ldots, X_n)\) is right-invariant.
    Condition \eqref{M12} implies that there are \(C',k >0\) such that
    \begin{equation}
        \label{ineq-2}
        \frac{1}{k}\varphi^*(k(t+1)) \leq \frac{1}{h}\varphi^*(ht) + \log C', \qquad t \geq 0.
    \end{equation}
    Let \(K\) be an arbitrary compact subset of \(G\). We apply Lemma \ref{MVT} to find $C>0$ such that \eqref{ineq-1} holds. 
    Let \(p \in \csn(E)\), \(\varepsilon \in (0,1]\), and \(v \in E^{(\omega)}\) be arbitrary.
    Since \(\bf{X}\) is right-invariant, we find that for all \(x \in G\) and \(\alpha \in \{1, \ldots, n \}^j\), \(j \in \N\),
    \begin{align*}
        X^\alpha(\gamma_v - \gamma_{\Pi(\chi_\varepsilon)v})(x)
        &= X^\alpha(\gamma_v - \gamma_v \ast \chi_\varepsilon)(x)
        = X^\alpha\gamma_v(x) - (X^\alpha \gamma_v)\ast \chi_\varepsilon(x) \\
        &= \int_G (X^\alpha\gamma_v(x) -X^\alpha \gamma_v(xy)) \check\chi_\varepsilon(y) dy.
    \end{align*}
    Hence, applying \eqref{ineq-1} to \(X^\alpha \gamma_v\), we obtain that
    \[
        \sup_{x \in K} p(X^\alpha(\gamma_v - \gamma_{\Pi(\chi_\varepsilon)v})(x)) 
        \leq \varepsilon C \sup_{x \in K}\sup_{y \in S_1} \max_{i \in \{ 1, \ldots, n\}} p(X_iX^\alpha\gamma_v(xy))
        .
    \]
Combining the latter inequality with \eqref{ineq-2}, we find that
    \[
        p_{{\mathbf{X}},K,\omega,h}(v-\Pi(\Molf_{\varepsilon})(v)) 
        \leq \varepsilon CC' p_{{\mathbf{X}},KS_{1},\omega,k}(v)
        .
        \qedhere
    \]
\end{proof}

\subsection{The parametrix method} \label{subs-parametrix} Let \(E\) be a sequentially complete lcHs.
Given a power series \(P(z) = \sum_{i =0}^\infty a_i z^i\), \(a_i \in \C\), we formally define
\[
    P(D) f 
    = \sum_{i =0}^\infty a_i f^{(i)}
    , \qquad f \in C^\infty(\R;E)
    .
\]
An entire function \(P(z) = \sum_{i =0}^\infty a_i z^i\) is called an \emph{ultrapolynomial of class }\((\omega)\) if there is \(H >0\) such that
\[
    \sup_{z \in \C} \abs{P(z)}e^{-H\omega(\abs{z})} 
    < \infty
    .
\]
In such a case, by using the Cauchy estimates, we find that
\[
    \sup_{i \in \N} \abs{a_i}\exp \left(H\varphi^*\left(\frac{i}{H}\right)\right) 
    < \infty
    .
\]
Hence, \eqref{M12} and the convexity of \(\varphi^*\) imply that for all \(h >0\) there is \(k >0\) such that
\begin{equation}
    \label{convUP}
    \sup_{j \in \N} \exp \left(-\frac{1}{h}\varphi^*(hj)\right) \sum_{i =0}^\infty \abs{a_i} \exp \left(\frac{1}{k}\varphi^*(k(i+j))\right) < \infty.
\end{equation}

Given an open set \(\Theta \subseteq \R^n\) and \(h >0\), we set \(\mathcal{D}^{\omega,h}(\Theta) = \mathcal{D}(\Theta) \bigcap \mathcal{E}^{\omega,h}(\R^n)\).
We denote by \(\mathcal{D}'(\R)\) (\(\mathcal{D}'^{(\omega)}(\R)\)) the dual of \(\mathcal{D}(\R)\) (\(\mathcal{D}^{(\omega)}(\R)\)).
We will assume the reader is familiar with the basic aspects about the spaces \(\mathcal{D}'(\R)\) and \(\mathcal{D}'^{(\omega)}(\R)\); see \cite{BMT,Schwartz} for more information.
The following standard result is fundamental for us.

\begin{proposition} \label{parametrix} 
    Let \(r >0\).
    \begin{enumerate}[i]
        \item \label{parametrix-1}
            For all \(j \in \N\) there are a polynomial \(P\) and \(\psi_0, \psi_1 \in C^j_c((-r,r))\) such that
            \(P(D) \psi_0 + \psi_1 = \delta \text{ in } \mathcal{D}'(\R)\).

        \item \label{parametrix-2}
            For all \(k' >0\) there are an ultrapolynomial \(P(z) = \sum_{i =0}^\infty a_i z^i\) of class \((\omega)\) and \(\psi_0, \psi_1 \in \mathcal{D}^{\omega,k'}((-r,r))\) such that
            \(P(D) \psi_0 + \psi_1 = \delta \text{ in } \mathcal{D}'^{(\omega)}(\R)\).
    \end{enumerate}
\end{proposition}
\begin{proof}
    \ref{parametrix-1} 
    Let \(H\) be the Heaviside function, that is, the indicator function of \([0,\infty)\).
    Choose \(\psi \in \mathcal{D}((-r,r))\) such that \(\psi =1\) on a neighborhood of \(0\).
    Then, \(P(z) = z^{j+2}\), \(\psi_0(x) = \frac{x^{j+1}}{(j+1)!}H(x) \psi(x)\), and \(\psi_1(x) = P(D)(\frac{x^{j+1}}{(j+1)!}H(x)(1-\psi(x)))\) verify all requirements. \\
    \ref{parametrix-2} 
    This is shown in \cite[Corollary 2.6]{Gomez-Collado}.
\end{proof}

Let \(\pi\) be a representation of \(G\) on a sequentially complete lcHs \(E\).
Let \({\bf{X}} = (X_1, \ldots, X_n)\) be a left-invariant frame on \(G\).
For \(h >0\), we define \(E^{\omega,h}_{\mathbf{X}}\) as the space of all \(v \in E^\infty\) such that \(p_{{\bf{X}},K,\omega,h}(v) < \infty\) for all \(p \in \operatorname{csn}(E)\) {and \(K \subseteq G\) compact}.
We endow \(E^{\omega,h}_{\mathbf{X}}\) with the locally convex topology generated by the system of seminorms \(\{ p_{\mathbf{X},K,\omega,h} \mid p \in \csn(E), K \subseteq G \text{ compact} \}\).
Note that \(E^{\omega,h}_{\mathbf{X}}\) is a sequentially complete lcHs.
Let \(P(z) = \sum_{i =0}^\infty a_i z^i\) be an ultrapolynomial of class \((\omega)\).
Let \(h >0\) be arbitrary and choose \(k >0\) such that \eqref{convUP} holds.
Then, the linear mapping
\[
  {\pi( P(X_j))}\colon E^{\omega,k}_{\mathbf{X}} \to E^{\omega,h}_{\mathbf{X}}, \quad   {\pi(P(X_j))} v = \sum_{i =0}^\infty a_i    {\pi(X_j^i)}v
\]
is continuous and the series \(   {\pi(P(X_j))} v= \sum_{i =0}^\infty a_i    {\pi(X_j^i)}v\) converges absolutely in \(E^{\omega,h}_{\mathbf{X}}\) for each \(v \in E^{\omega,k}_{\mathbf{X}}\) and each  \(j =1, \ldots,n\).
Our goal in this subsection is to show the following parametrix type result.

\begin{theorem}\label{par-final} 
   Let \(\pi\) be a representation of \(G\) on a sequentially complete lcHs \(E\), let \({\bf{X}} = (X_1, \ldots, X_n)\) be a left-invariant frame on \(G\), and let \(U\) be an open neighborhood of \(e\).
    \begin{enumerate}[i]
        \item \label{par-final-1}
            For all \(j \in \N\) there are a polynomial \(P\) and \(\chi_\theta \in C^j_c(U)\), \(\theta = (\theta_1, \ldots, \theta_n) \in \{0,1\}^n\), such that for all \(v \in E^\infty\)
            \begin{equation}
                \label{par-G}
                v = \sum_{\theta \in \{0,1\}^n} \Pi(\chi_\theta)(   {\pi(P_{\theta_n}(X_n)) \cdots \pi(P_{\theta_1}(X_1))v)},
            \end{equation}
            where \(P_0 = P\) and \(P_1 =1\).
        \item \label{par-final-2} 
            For all \(h,h' >0\) there are \(k >0\), an ultrapolynomial \(P\) of class \((\omega)\), \(\chi_\theta \in \mathcal{D}(U)\) with \(\norm{\chi_\theta}_{{\mathbf{X}},\omega,h'} < \infty\) , \(\theta \in \{0,1\}^n\) such that
            \eqref{par-G} holds for all \(v \in E^{\omega,k}_{\mathbf{X}}\), where \(P_0 = P\) and \(P_1 =1\), and the linear mapping
            \[
             {  \pi(  P_{\theta_n}(X_n)) \cdots\pi( P_{\theta_1}(X_1))}\colon E^{\omega,k}_{\mathbf{X}} \to E^{\omega,h}_{\mathbf{X}}
            \]
            is continuous for each \(\theta \in \{0,1\}^n\).
    \end{enumerate}
\end{theorem}

 We will show \Cref{par-final} by combining \cref{parametrix} with the aid of a technique due to Dixmier and Malliavin \cite{D-M}.
We need various notions and results in preparation.

\begin{lemma} \label{par-cor-eucl} 
    Let \(E\) be a sequentially complete lcHs.
    \begin{enumerate}[i]
        \item \label{par-cor-eucl-1}
            Let \(P\) be a polynomial and \(\psi_0,\psi_1 \in C_c(\R)\) such that \(P(D) \psi_0 + \psi_1 = \delta \text{ in } \mathcal{D}'(\R)\).
            Then, for all \(f \in C^\infty(\R;E)\)
            \begin{equation}
                \label{decomp-0}
                f = (P(D)f) \ast \psi_0 + f \ast \psi_1.
            \end{equation}
        \item \label{par-cor-eucl-2} 
            Let \(P(z) = \sum_{i =0}^\infty a_i z^i\) be an ultrapolynomial of class \((\omega)\) and \(\psi_0,\psi_1 \in C_c(\R)\) such that \(P(D) \psi_0 + \psi_1 = \delta \text{ in } \mathcal{D}'^{(\omega)}(\R)\).
            Let \(f \in C^\infty(\R;E)\) be such that \(P(D)f = \sum_{i = 0}^\infty a_i f^{(i)}\) converges absolutely in \(C(\R;E)\).
            Then, \eqref{decomp-0} holds.
    \end{enumerate}
\end{lemma}
\begin{proof}
    We only show \ref{par-cor-eucl-2} as the proof of \ref{par-cor-eucl-1} is similar.
    Recall that \(f_{v'} = \Eval{v'}{f} \in C^\infty(\R)\) for \(v' \in E'\).
    Note that \(P(D)f_{v'} = \sum_{n = 0}^\infty a_n f^{(n)}_{v'}\) converges absolutely in \(C(\R)\) and that \(P(D)f_{v'} = \Eval{v'}{P(D)f}\).
    In \(\mathcal{D}'^{(\omega)}(\R)\), we have 
    \[
        f_{v'} 
        = f_{v'} \ast \delta 
        = f_{v'} \ast (P(D) \psi_0) + f_{v'} \ast \psi_1 
        = (P(D) f_{v'}) \ast \psi_0+ f_{v'} \ast \psi_1
        .
    \]
    Since \(P(D)f_{v'}, f_{v'} \in C(\R)\) and \(\psi_0, \psi_1 \in C_c(\R)\), the equality \(f_{v'} = (P(D) f_{v'}) \ast \psi_0+ f_{v'} \ast \psi_1\) in fact holds pointwise.
    We have
    \[
        (P(D) f_{v'}) \ast \psi_0 
        = \Eval{v'}{P(D)f} \ast \psi_0 
        = \Eval{v'}{P(D) f \ast \psi_0}
    \]
    and
    \[
        f_{v'} \ast \psi_1 
        = \Eval{v'}{f \ast \psi_1}
        .
    \]
    Consequently, we find that, for all \(v' \in E'\),
    \[
        \Eval{v'}{f} 
        = f_{v'} 
        = (P(D) f_{v'}) \ast \psi_0+ f_{v'} \ast \psi_1 
        = \Eval{v'}{(P(D) f) \ast \psi_0 + f \ast \psi_1}
        ,
    \]
    which implies that \(f = (P(D) f) \ast \psi_0 + f \ast \psi_1\) by the Hahn-Banach theorem.
\end{proof}

We denote by \(\mathfrak{g}\) the Lie algebra of \(G\) and by \(\exp\colon \mathfrak{g} \to G\) the exponential mapping.
We identify each \(X \in \mathfrak{g}\) with its associated left-invariant vector field on \(G\), that is,
\begin{equation}
    \label{equ:associated-left-invariant-vector-field}
    Xf(x) =
\left.\frac{d}{ dt}\right|_{t =0}f(x\exp(tX))
    , \qquad f \in C^\infty(G), x \in G.
\end{equation}
Let \(\pi\) be a representation of \(G\) on a sequentially complete lcHs \(E\).
For \(X \in \mathfrak{g}\), we define the representation
\[
    \pi_X\colon (\R, +) \to \operatorname{GL}(E), \qquad \pi_X(t) = \pi(\exp(tX)).
\]
We denote the orbit of \(v \in E\) under \(\pi_X\) by \(\gamma_{X,v}\).
Note that \(\gamma_{X,v}(t) = \gamma_{v}(\exp(tX))\).
Hence, \(\gamma_{X,v} \in C(\R;E)\) (\(\gamma_{X,v} \in C^\infty(\R;E)\)) if \(v \in E^0\) (\(v \in E^\infty\)).
In accordance to \cref{Section smooth and ultradifferentiable vectors}, we set
\[
    \Pi_X(f) v 
    = \int_{\R} \gamma_{X,v}(t) f(-t) dt
    , \qquad f \in C_c(\R), v \in E^0
    .
\]
It holds that
\[
    \gamma_{X, \Pi_X(f) v} 
    = \gamma_{X,v} \ast f
    , \qquad f \in C_c(\R), v \in E^0
    .
\]
In view of \eqref{eq:infinitesimalaction} and  \eqref{orbit-prop},
\[
\pi(X^{i})v= \pi_{X}\left(\frac{d^{i}}{dt^{i}}\right)v = X^i\gamma_v(e) \qquad \mbox{and}\qquad  \gamma_{X, \pi(X^{i})v} 
    = \gamma^{(i)}_{X,v}
    , \qquad v \in E^\infty
    .
\]

Given a power series \(P(z) = \sum_{i =0}^\infty a_i z^i\), we formally write
\[
 {\pi (P(X)) v = \pi_{X}\Big(P\Big(\frac{d}{dt}\Big)\Big)
    = \sum_{i =0}^\infty a_i \pi(X^i)v}
    , \qquad v \in E^\infty
    .
\]

\begin{lemma} \label{par-cor} 
    Let \(X \in \mathfrak{g}\) and let \(\pi\) be a representation of \(G\) on a sequentially complete lcHs \(E\).
    \begin{enumerate}[i]
        \item Let \(P\) be a polynomial and \(\psi_0,\psi_1 \in C_c(\R)\) such that \(P(D) \psi_0 + \psi_1 = \delta \text{ in } \mathcal{D}'(\R)\).
            Then, for all \(v \in E^\infty\)
            \begin{equation}
                \label{decomp-1}
                v = \Pi_{X}(\psi_0){\pi(P(X))}v + \Pi_{X}(\psi_1)v.
            \end{equation}
        \item Let \(P(z) = \sum_{i =0}^\infty a_i z^i\) be an ultrapolynomial of class \((\omega)\) and \(\psi_0,\psi_1 \in C_c(\R)\) such that \(P(D) \psi_0 + \psi_1 = \delta \text{ in } \mathcal{D}'^{(\omega)}(\R)\).
            Let \(v \in E^\infty\) be such that {\(P(D) \gamma_{X,v} = \sum_{i = 0}^\infty a_i \gamma_{X,v}^{(i)} \)} converges absolutely in {\(C(\R;E)\)}.
            Then, \eqref{decomp-1} holds.
    \end{enumerate}
\end{lemma}
\begin{proof} We only show $(ii)$ as the proof of $(i)$ is similar.
    {By assumption,}
    \[
        P(D) \gamma_{X,v} 
        = \sum_{i = 0}^\infty a_i \gamma_{X,v}^{(i)} 
        = \sum_{i = 0}^\infty a_i \gamma_{X,{\pi(X^i)v}}
    \]
    converges absolutely in \(C(\R;E)\).
    Moreover, one also has \(P(D) \gamma_{X,v} = \gamma_{X,   {\pi(P(X))}v}\).
    Hence, \Cref{par-cor-eucl} implies that
    \[
        \gamma_{X,v} 
        = \gamma_{X,{\pi(P(X))}} \ast \psi_0 + \gamma_{X,v} \ast \psi_1 
        = \gamma_{X, \Pi_{X}(\psi_0)   {\pi(P(X))}v} + \gamma_{X, \Pi_{X}(\psi_1)v}
        .
    \]
    The result now follows by evaluating the above equality at \(0\).
\end{proof}

Let \({\bf{X}} = (X_1, \ldots, X_n)\) be a basis of \(\mathfrak{g}\) and consider the mapping (exponential coordinates of the second type)
\[
    \Phi\colon \R^n \to G, \quad \Phi(t_1, \ldots, t_n) = \exp(t_{1}X_{1}) \cdots \exp(t_{n}X_{n}).
\]
Then, there is \(r_0 >0\) such \(\Phi\colon (-r_0,r_0)^{n} \to G\) is a regular diffeomorphism onto its image.

\begin{lemma}\label{change-of-var}
    Let \({\bf{X}}\) be a basis of \(\mathfrak{g}\) and let \(\Phi,r_0\) be as above, let \(\pi\) be a representation of \(G\) on a sequentially complete lcHs \(E\), and let \(r < r_0\).
    \begin{enumerate}[i]
        \item \label{change-of-var-1}
            Let \(j \in \N\).
            For all
            \(\psi_1, \ldots, \psi_n \in C^j_c((-r,r))\) there is \(\chi \in C^j(\Phi((-r,r)^n))\) such that
            \begin{equation}
                \label{DM-trick}
                \Pi_{X_1}(\psi_1) \cdots \Pi_{X_n}(\psi_n) v= \Pi(\check{\chi})v, \qquad v \in E^0.
            \end{equation}
        \item \label{change-of-var-2}
            For all \(h'>0\) there is \(k'>0\) such that for all \(\psi_1, \ldots, \psi_n \in \mathcal{D}^{\omega,k'}((-r,r))\) there is \(\chi \in \mathcal{D}(\Phi((-r,r)^n))\) with \(\chi \circ \Phi \in \mathcal{D}^{\omega,h'}((-r,r)^n)\) such that
            \eqref{DM-trick} holds.
    \end{enumerate}
\end{lemma}
\begin{proof}
    We only show \ref{change-of-var-2} as the proof of \ref{change-of-var-1} is similar.
    Condition \eqref{M12} yields that there are \(C,k' >0\) such that
    \[
        \frac{1}{k'}\varphi^*(k't) + (\log 2)t 
        \leq \frac{1}{h'}\varphi^*(h't) + \log C
        , \qquad t \geq 0
        .
    \]
    This inequality together with the convexity of \(\varphi^*\) implies that
    \begin{equation}
        \label{mult}
        f \psi \in \mathcal{D}^{\omega, h'}((-r,r)^n), \qquad f \in \mathcal{E}^{\omega,k'}((-r,r)^n), \psi \in \mathcal{D}^{\omega,k'}((-r,r)^n).
    \end{equation}
    Let \(\psi_1, \ldots, \psi_n \in \mathcal{D}^{\omega,k'}((-r,r))\) be arbitrary.
    Then \(\psi_1 \otimes \cdots \otimes \psi_n \in \mathcal{D}^{\omega,k'}((-r,r)^n)\) because \(\varphi^*\) is convex.
    Note that there is a real analytic function \(J\colon \Phi((-r_0,r_0)^n) \to \R\) such that
    \[
        \int_{\R^n} f(t) \psi(t) dt 
        = \int_{G} f(\Phi^{-1}(x)) \psi(\Phi^{-1}(x)) J(x) dx
        , \quad f \in C(\R^n;E), \psi \in \mathcal{D}((-r_0,r_0)^n)
        .
    \]
    Set \(\chi(x) = (\psi_1 \otimes \cdots \otimes \psi_n)(-\Phi^{-1}(x))J(x)\), \(x \in G\).
    Since \(\psi_1 \otimes \cdots \otimes \psi_n \in \mathcal{D}^{\omega,k'}((-r,r)^n)\) and \(J \circ \Phi \in \mathcal{A}((-r,r)^n) \subseteq \mathcal{E}^{\omega,k'}((-r,r)^n)\), \eqref{mult} implies that \(\chi \circ \Phi \in \mathcal{D}^{\omega,h'}((-r,r)^n)\).
    For each \(v \in E^0\),
    \begin{align*}
        \Pi_{X_1}(\psi_1) \cdots \Pi_{X_n}(\psi_n) v
        & = \int_{\R^n} \gamma_v(\Phi(t)) (\psi_1 \otimes \cdots \otimes \psi_n)(-t) dt \\
        &= \int_{G} \gamma_v(x)(\psi_1 \otimes \cdots \otimes \psi_n)(-\Phi^{-1}(x)) J(x) dx \\
        &= \int_{G} \gamma_v(x) \chi(x) dx = \Pi(\check{\chi})v
        .
        \qedhere
    \end{align*}
\end{proof}

\begin{proof}[Proof of \Cref{par-final}]
    This follows from \Cref{parametrix}, and \Cref{par-cor,change-of-var} (for \ref{par-final-2} we also use Proposition \ref{switch-explicit}\ref{switch-explicit-2}).
\end{proof}

\begin{remark} 
    Cartier also employed the parametrix method to study smooth vectors in \cite{C}. 
    Our applications to be discussed in the following two sections will require the more explicit form we have treated in this subsection.
\end{remark}

\section{Main results}
\label{sect-main}
We are ready to study quasinormability and the property \((\Omega)\) for the spaces \(E^{\infty}\) and \(E^{(\omega)}\).

\subsection{Quasinormability} Given a lcHs \(F\), we denote by \(\mathcal{U}_0(F)\) the family
of all absolutely convex neighborhoods of \(0\) in \(F\) and by \(\mathcal{B}(F)\) the family  of all bounded subsets of
\(F\).
Given \(p \in \csn(F)\) we write \(V_p = \{ v \in F \mid p(v) \leq 1 \} \in \mathcal{U}_0(F)\).
The space \(F\) is said to be quasinormable \cite[p.\ 313]{M-V} if
\[
    \forall U \in \mathcal{U}_0(F) \, \exists V \in \mathcal{U}_0(F) \, \forall \varepsilon \in (0,1] \, \exists B \in \mathcal{B}(F) \colon 
    V \subseteq \varepsilon U + B
    .
\]

We are ready to prove the first main result of this article.
\begin{theorem} \label{main}
    Let \(\pi\) be a {continuous} {and} locally equicontinuous representation of \(G\) on a sequentially complete lcHs \(E\).
    \begin{enumerate}[i]
        \item \label{main-1}
            \(E^\infty\) is quasinormable if \(E\) is so.
        \item \label{main-2} 
            \(E^{(\omega)}\) is quasinormable if \(E\) is so.
    \end{enumerate}
\end{theorem}
\begin{proof}
    We only show \ref{main-2} as the proof of \ref{main-1} is similar.
    Let \(\bf{X}\) be a {left-invariant}  frame on \(G\).
 Employing Proposition \ref{main-frames}\ref{main-frames-2}, it suffices to show that
    \begin{gather*}
        \forall K \subseteq \LieGroup \text{ compact}, h>0, p \in \csn(E) \, \exists L \subseteq \LieGroup \text{ compact}, k >0, q \in \csn(E) \\
        \forall \varepsilon \in (0,1] \, \exists B \in \mathcal{B}(E^{(\omega)}) \colon
        V_{q_{{\bf{X}},L,\omega, k}} \subseteq \varepsilon V_{p_{{\bf{X}},K,\omega, h}} + B
        .
    \end{gather*}
    Let \(K \subseteq \LieGroup\) compact, \(h >0\), and \(p \in \csn(E)\) be arbitrary.
    Take \((\Molf_{\varepsilon})_{\varepsilon \in (0,1]} \subseteq \mathcal{D}^{(\omega)}(G)\) as in \Cref{approx}.
    Set \(L = KS_1\).
    Hence, there are \(C,k >0\) such that for all \(\varepsilon \in (0,1]\)
    \[
        p_{{\mathbf{X}},K,\omega,h}(v-\Pi(\Molf_{\varepsilon})(v)) 
        \leq \varepsilon C p_{{\mathbf{X}},L,\omega,k}(v)
        , \qquad v \in E^{(\omega)}
        .
    \]
    Furthermore, \Cref{smoothening}\ref{smoothening-2}
     implies that there is \(r \in \csn(E)\) such that for all \(\varepsilon \in (0,1]\)
    \[
        p_{{\bf{X}},K,\omega,h}(\Pi(\chi_\varepsilon)v) 
        \leq \norm{ \chi_\varepsilon }_{{\bf{X}},\omega,h} r(v)
        , \qquad v \in E^0 = E
        .
    \]
  
    Since \(E\) is quasinormable, there is \(s \in \csn(\Vecs)\) such that
    \begin{equation}
        \label{equ:quasinormability-vecs}
        \forall \delta > 0 \, \exists A \in \mathcal{B}(\Vecs) \colon V_s \subseteq \delta V_r + A.
    \end{equation}
    Set \(q = \max \{s,Cp\} \in \csn(E)\).
    Let \(\varepsilon \in (0,1]\) be arbitrary.
    Choose \(A \in \mathcal{B}(E)\) according to \eqref{equ:quasinormability-vecs} with \(\delta = \varepsilon / (2\norm{ \chi_{\varepsilon/2} }_{{\bf{X}},\omega,h})\).
    Set \(B = \{\Pi(\Molf_{\varepsilon/2})(u)\mid u \in A\} \).
    Note that \(B \in \mathcal{B}(E^{(\omega)})\) by \Cref{smoothening}\ref{smoothening-2}.
    Let \(v \in V_{q_{{\bf{X}},L,\omega, k}}\) be arbitrary.
    There is some \(u \in A\) such that \(r(v-u) \leq \varepsilon / (2\norm{ \chi_{\varepsilon/2} }_{{\bf{X}},\omega,h})\).
    For \(w = \Pi(\Molf_{\varepsilon/2})(u) \in B\), we obtain that
    \begin{align*}
        p_{{\bf{X}},K,\omega, h}(v -w) 
        &\leq p_{{\bf{X}},K,\omega,h}(v-\Pi(\Molf_{\varepsilon/2})(v)) + p_{{\bf{X}},K,\omega,h}(\Pi(\Molf_{\varepsilon/2})(v-u)) \\
        &\leq \frac{\varepsilon C}{2}p_{{\mathbf{X}},L,\omega,k}(v) + \norm{ \chi_{\varepsilon/2} }_{{\bf{X}},\omega,h} r(v-u) 
        \leq \varepsilon
        ,
    \end{align*}
    which shows the result.
\end{proof}

{
\begin{remark}
    \label{rmk:extension-action-quasinormability}
    One may weaken the assumptions from \Cref{main} and instead assume that \(\Rep\) is an arbitrary representation for which \(\Pi \colon C_{c}(G) \times E^{0} \to E^{0}\) extends to a bilinear map \(C_{c}(G) \times E \to E\) such that the statements from \Cref{smoothening} hold for all \(v \in E\).
 The proof works exactly in the same way under this more general assumption.
\end{remark}
}

\subsection{The condition \texorpdfstring{\((\Omega)\)}{(Ω)} } 
A Fr\'{e}chet space \(F\) is said to satisfy the condition \((\Omega)\) \cite[p.\ 367]{M-V} if
\begin{gather*}
    \forall U \in \mathcal{U}_0(F) \, \exists V \in \mathcal{U}_0(F) \, \forall \, W \in \mathcal{U}_0(F) \, \exists C, s > 0 \, \forall \varepsilon \in (0,1] \colon
    V \subseteq \varepsilon U + \frac{C}{\varepsilon^s}W
    .
\end{gather*}
\begin{remark}\label{QN-F}
In view of  \cite[Lemma 26.14]{M-V}, a Fr\'{e}chet space \(F\) is quasinormable if and only if
    \begin{gather*}
        \forall U \in \mathcal{U}_0(F) \, \exists V \in \mathcal{U}_0(F) \, \forall \, W \in \mathcal{U}_0(F) \, \forall \varepsilon \in (0,1] \, \exists C> 0 \colon 
        V \subseteq \varepsilon U + CW
        .
    \end{gather*}
    Hence, \((\Omega)\) may be considered as a quantified version of quasinormability.
\end{remark}

The second main result of this article reads as follows.
\begin{theorem} \label{main-omega}
    Let \(\pi\) be a  continuous representation of \(G\) on a Fr\'{e}chet space \(E\).
    \begin{enumerate}[i]
        \item \label{main-omega-1}
            \(E^\infty\) satisfies \((\Omega)\) if \(E\) does so.
        \item \label{main-omega-2} 
            \(E^{(\omega)}\) satisfies \((\Omega)\) if \(E\) does so.
    \end{enumerate}
\end{theorem}
Since the space $E$ in \Cref{main-omega} is Fr\'echet,  the continuous representation $\pi$ is automatically locally equicontinuous, cf.\ \Cref{rmk:continuity-implies-local-equicontinuity}.
We first prove \Cref{main-omega}\ref{main-omega-1}.
This shall be done by using a refinement of the technique employed in the proof of \Cref{main}.
We write
\[
    \norm{ f}_{j} 
    = \sup_{t \in \R^n} \max_{\abs{\alpha} \leq j} {\abs{f^{(\alpha)}(t)}}
    , \qquad f \in C^\infty(\R^n), \ j \in \N
    .
\]

\begin{proof}[Proof of \Cref{main-omega}\ref{main-omega-1}] 
    We use the same notation as in Subsection \ref{subs-approx}.
    Let \(\bf{X}\) be a smooth frame on \(G\).
    Pick \(\psi \in \mathcal{D}(B(0,1))\) such that \(\psi \geq 0\) and \(\int_{\R^n} \psi(t)dt = 1\).
    Set \(\psi_\varepsilon(t) =\varepsilon^{-n} \psi(t/\varepsilon)\) for \(\varepsilon \in (0,1]\).
    Define \(\chi_\varepsilon \in \mathcal{D}(G)\) via
    \[
        \check{\chi}_\varepsilon 
        = \frac{\psi_\varepsilon \circ \phi}{ \norm{\psi_\varepsilon \circ \phi}_{L^1(G)} }
        .
    \]
    Then, \((\chi_\varepsilon)_{\varepsilon \in (0,1]}\) satisfies the assumptions of \Cref{approx}.
    There is a positive smooth function \(J\) on \(U\) such that
    \[
        \int_{\R^n} f(t) dt 
        = \int_{G} f(\phi(x)) J(x) dx
        , \qquad f \in \mathcal{D}(\phi(U))
        .
    \]
    Hence, we obtain that for all \(\varepsilon \in (0,1]\)
    \[
        \norm{\psi_\varepsilon \circ \phi}_{L^1(G)} 
        \geq \frac{1}{\sup_{g \in \phi^{-1}(\overline{B}(0,1))} J(g)}
        .
    \]
    Proposition \ref{switch-explicit}\ref{switch-explicit-1} therefore implies that for all \(j \in \N\) there is \(C_j>0\) such that for all \(\varepsilon \in (0,1]\)
    \begin{equation}
        \label{boundes-eps}
        \norm{ \chi_\varepsilon}_{{\bf{X}},j} 
        \leq C_j \norm{ \psi_\varepsilon }_j 
        \leq \frac{C_j \norm{ \psi}_j}{\varepsilon^{j+n}}.
    \end{equation}
    We are ready to show that \(E^\infty\) satisfies \((\Omega)\).
    By Proposition \ref{main-frames}\ref{main-frames-1} and rescaling, it suffices to show that
    \begin{gather*}
        \forall K \subseteq \LieGroup \text{ compact}, j \in \N, p \in \csn(E) \, \exists  q \in \csn(E) \, \forall L \subseteq \LieGroup \text{ compact}, m \in \N, r \in \csn(E) \\
        \exists C_1, C_2, s > 0 \, \forall \varepsilon \in (0,1] \colon V_{q_{{\bf{X}},KS_1,j+1}} \subseteq \varepsilon C_1V_{p_{{\bf{X}},K,j}} + \frac{C_2}{\varepsilon^s}V_{r_{{\bf{X}},L,m}}
        .
    \end{gather*}
    Let \(K \subseteq \LieGroup\) compact, \(j \in \N\), \(p \in \csn(E)\) be arbitrary.
    \Cref{approx}\ref{approx-1} implies that there is \(C >0\) such that
    \[
        p_{{\mathbf{X}},K,j}(v - \Pi(\chi_\varepsilon)v) 
        \leq \varepsilon Cp_{{\mathbf{X}},KS_1,j+1}(v)
        , \qquad v \in E^\infty
        .
    \]
    By \Cref{smoothening}\ref{smoothening-1}, there is \(p' \in \csn(E)\) such that for all \(\varepsilon \in (0,1]\)
    \[
        p_{\mathbf{X},K,j}(\Pi(\chi_\varepsilon)v) 
        \leq \norm{ \chi_\varepsilon }_{{\bf{X}},j} p'(v)
        , \qquad v \in E^0 =E
        .
    \]
    Since \(E\) satisfies \((\Omega)\), there is \(q \in \csn(E)\), \(q \geq \max \{ p,p'\}\), such that
    \begin{equation}
        \label{om-s}
        \forall r' \in \operatorname{csn}(E) \, \exists C', s' > 0 \, \forall \delta \in (0,1] \colon V_q \subseteq \delta V_{p'} + \frac{C'}{\delta^{s'}}V_{r'}.
    \end{equation}
    Let \(L \subseteq \LieGroup\) compact, \(m \in \N\), and \(r \in \csn(E)\) be arbitrary.
\Cref{smoothening}\ref{smoothening-1} implies that there is \(r' \in \csn(E)\) such that for all \(\varepsilon \in (0,1]\)
    \[
        r_{\mathbf{X},L,m}(\Pi(\chi_\varepsilon)v) 
        \leq \norm{ \chi_\varepsilon }_{{\bf{X}},m} r'(v)
        , \qquad v \in E^0 =E
        .
    \]
    Choose \(C',s'>0\) according to  \eqref{om-s}.
    Let \(v \in V_{q_{{\bf{X}},KS_1,j+1}}\) be arbitrary.
    Using \eqref{om-s}, one obtains that for all \(\delta \in (0,1]\) there is \(w_{\delta} \in C'\delta^{-s'} V_{r'}\) such that \(v-w_\delta \in \delta V_{p'}\). We have, 
    for all \(\varepsilon,\delta \in (0,1]\),
    \begin{align*}
        p_{\mathbf{X},K,j}(v - \Pi(\chi_\varepsilon)w_{\delta}) 
        &\leq p_{\mathbf{X},K,j}(v - \Pi(\chi_\varepsilon)v) + p_{\mathbf{X},K,j}(\Pi(\chi_\varepsilon)(v-w_{\delta})) \\
        &\leq \varepsilon C+ \norm{ \chi_\varepsilon }_{{\bf{X}},j} p'(v-w_\delta) \\
        &\leq \varepsilon C+ \frac{C_j \norm{ \psi }_j }{\varepsilon^{j+n}} \delta
    \end{align*}
    and
    \[
        r_{\mathbf{X},L,m}(\Pi(\chi_\varepsilon)w_{\delta}) 
        \leq \norm{ \chi_\varepsilon }_{{\bf{X}},m} r'(w_\delta)
        \leq \frac{C'C_m\norm{ \psi }_m }{\varepsilon^{m+n} \delta^{s'}}
        .
    \]
    We set \(s = m +n + (j +n+1)s'\) and \(\delta_\varepsilon = \varepsilon^{j +n+1}\) for \(\varepsilon \in (0,1]\).
    Then, for all \(\varepsilon \in (0,1]\),
    \[
        v 
        = (v - \Pi(\chi_\varepsilon)w_{\delta_{\varepsilon}}) + \Pi(\chi_\varepsilon)w_{\delta_\varepsilon} 
        \in \varepsilon(C +C_j \norm{ \psi }_j ) V_{p_{{\bf{X}},K,j}} + \frac{C'C_m\norm{ \psi }_m}{\varepsilon^s}V_{r_{{\bf{X}},L,m}}
        .
        \qedhere
    \]
\end{proof}

Next, we show \Cref{main-omega}\ref{main-omega-2}.
Our proof is based on the parametrix method presented in Subsection \ref{subs-parametrix}.
We need two results in preparation.

\begin{lemma}\label{abstract}
    Let \(F_1 \supseteq F_2 \supseteq \cdots\) be a decreasing sequence of Fr\'{e}chet spaces with continuous inclusion mappings.
    Set \(F = \bigcap_{j \in \N} F_j\) and endow \(F\) with its natural Fr\'{e}chet space topology, namely, the initial topology with respect to the inclusion mappings \(F \to F_{j}\), \(j \in \N\).
    Suppose that
    \begin{gather*}
        \forall i \in \N, U \in \mathcal{U}_0(F_i) \, \exists j \geq i, V \in \mathcal{U}_0(F_j) \, \forall \, m \geq j, W \in \mathcal{U}_0(F_m) \\
        \exists C, s > 0 \, \forall \varepsilon \in (0,1] \colon V \subseteq \varepsilon U + \frac{C}{\varepsilon^s}W
        .
    \end{gather*}
    Then, \(F\) satisfies \((\Omega)\).
\end{lemma}
\begin{proof}
    This follows from the Mittag-Leffler theorem \cite[Theorem 3.2.8]{Wengenroth}, see \cite[Lemma 2.4]{D-K} for details.
\end{proof}

Our next step is to show a quantified approximation result for compactly supported ultradifferentiable functions on \(G\).
We first prove an analogous result for periodic ultradifferentiable functions on \(\R^n\).
Let \(a >0\).
We denote by \(C^\infty_a(\R^n)\) the Fr\'{e}chet space consisting of all smooth \(a\Z^n\)-periodic functions.
For \(h >0\) we write
\[
    \norm{ f}_{\omega,h} 
    = \sup_{t \in \R^n} \sup_{\alpha \in \N^n} \abs{f^{(\alpha)}(t)}\exp \left(-\frac{1}{h}\varphi^*(h\abs{\alpha})\right)
    , \qquad f \in C^\infty(\R^n)
    .
\]
\begin{lemma}\label{per}
    Let \(a >0\).
    Then,
    \begin{gather*}
        \forall h >0 \, \exists k >0 \, \forall l >0 \, \exists C,s > 0 \, \forall f \in C^\infty_a(\R^n), \norm{ f }_{\omega,k} \leq 1 \\
        \forall \varepsilon \in (0,1] \, \exists f_\varepsilon \in C^\infty_a(\R^n) \colon 
        \norm{ f- f_\varepsilon }_{\omega,h} \leq \varepsilon \quad \text{and} \quad \norm{ f_\varepsilon }_{\omega,l} \leq \frac{C}{\varepsilon^{s}}
        .
    \end{gather*}
\end{lemma}
\begin{proof}
    Define the Fr\'{e}chet space
    \[
        s(\Z^n) 
        = \{ c = (c_\alpha)_{\alpha \in \Z^n} \in \C^{\Z^n} \mid \sup_{\alpha \in \Z^n} \abs{c_\alpha} \abs{\alpha}^k < \infty , \, \forall k \in \N \}
        .
    \]
    Given \(f \in C^\infty_a(\R^n)\), we define its Fourier coefficients as
    \[
        \widehat{f}(\alpha) 
        = \frac{1}{a} \int_0^a f(t) e^{-\frac{2\pi i}{a} \alpha \cdot t } dt
        , \qquad \alpha \in \Z^n
        .
    \]
    Then
    \[
        \mathcal{F} \colon C^\infty_a(\R^n) \to s(\Z^n), \, f \mapsto \widehat{f} = (\widehat{f}(\alpha))_{\alpha \in \Z^n}
    \]
    is a topological isomorphism whose inverse is given by
    \begin{equation}
        \label{inverseFT}
        \mathcal{F}^{-1}(c)(t) = \sum_{\alpha \in \Z^{n}} c_\alpha e^{\frac{2\pi i}{a} \alpha \cdot t }, \qquad c = (c_\alpha)_{\alpha \in \Z^n} \in s(\Z^n),
    \end{equation}
    and the series on the right-hand side is absolutely convergent in \(C^\infty_a(\R^n)\).
    For \(h >0\), we define
    \[
        \abs{c}_{\omega,h} 
        = \sup_{\alpha \in \Z^n} \abs{c_\alpha} e^{\frac{1}{h} \omega(\abs{\alpha})}
        , \qquad c \in s(\Z^n)
        .
    \]
Using \eqref{M12}, a standard argument shows that for all \(h > 0\) there are \(C,k >0\) such that
    \[
        \abs{ \widehat{f} }_{\omega,h} 
        \leq C\norm{ f}_{\omega,k}
        , \qquad f \in C^\infty_a(\R^n)
        ,
    \]
    and that for all \(h > 0\) there are \(C,k >0\) such that
    \[
        \norm{ \mathcal{F}^{-1}(c) }_{\omega,h} 
        \leq C\abs{ c}_{\omega,k}
        , \qquad c \in s(\Z^n)
        .
    \]
    Since \(\mathcal{F}\colon C^\infty_a(\R^n) \to s(\Z^n)\) is an isomorphism whose inverse is given by \eqref{inverseFT}, it therefore suffices to show that
    \begin{gather*}
        \forall h >0 \, \exists k >0 \, \forall l >0 \, \exists C,s > 0 \, \forall c \in s(\Z^n), \abs{ c }_{\omega,k} \leq 1 \\
        \forall \varepsilon \in (0,1] \, \exists c_\varepsilon \in s(\Z^n) \colon 
        \abs{ c- c_\varepsilon }_{\omega,h} \leq \varepsilon \quad \text{and} \quad \abs{ c_\varepsilon }_{\omega,l} \leq \frac{C}{\varepsilon^{s}}
        .
    \end{gather*}
    Let \(h >0\) be arbitrary.
    Set \(k =h/2\).
    Let \(l >0\) be arbitrary.
    Let \(c \in s(\Z^n)\) with \(\abs{ c }_{\omega,k} \leq 1\) be arbitrary.
    For \(\varepsilon \in (0,1]\), we choose a number \(N_{\varepsilon}
    \in [0,\infty)\) such that \(\omega(N_{\varepsilon}
    ) = h \log(1/\varepsilon)\).
    We define \(c_\varepsilon = (c_{\varepsilon,\alpha})_{\alpha \in \Z^n} \in s(\Z^n)\) by \(c_{\varepsilon,\alpha} = c_\alpha\) if \(\abs{\alpha} \leq N_{\varepsilon}\) and \(c_{\varepsilon,\alpha} = 0\) otherwise.
    Set \(s = h/l\).
    Then, it holds that \(\abs{ c- c_\varepsilon }_{\omega,h} \leq \varepsilon\) and \(\abs{ c_\varepsilon }_{\omega,l} \leq \varepsilon^{-s}\) for all \(\varepsilon \in (0,1]\) .
\end{proof}
\begin{lemma}\label{decomp-G}
    Let \({\bf{X}}\) be a real analytic frame on \(G\).
    Let \(K\) and \(L\) be compact subsets of \(G\)
    such that \(K \subseteq \operatorname{int} L\).
    Then,
    \begin{gather*}
        \forall h >0 \, \exists k >0 \, \forall l >0 \, \exists C,s > 0 \,
        \forall f \in \mathcal{D}_K,  \norm{ f }_{\mathbf{X},\omega,k} \leq 1 \\
        \forall \varepsilon \in (0,1] \, \exists f_\varepsilon \in \mathcal{D}_L ~ \colon
        \norm{ f-f_\varepsilon }_{\mathbf{X},\omega,h} \leq \varepsilon 
        \quad \text{and} \quad 
        \norm{ f_\varepsilon }_{\mathbf{X},\omega,l} \leq \frac{C}{\varepsilon^{s}}
        .
    \end{gather*}
\end{lemma}
\begin{proof}
    By using local coordinates, a partition of the unity argument, and Proposition \ref{switch-explicit}\ref{switch-explicit-2}, we may assume that \(G = \R^n\) and \({\bf{X}} = \{ \partial_1, \ldots, \partial_n\}\).
    Choose \(a > 0\) such that \(L \subseteq (-a/2,a/2)^n\).
    Let \(h >0\) be arbitrary.
    Condition \eqref{M12} yields that there are \(C',h' >0\) such that
    \[
        \frac{1}{h'}\varphi^*(h't) + (\log 2)t 
        \leq \frac{1}{h}\varphi^*(ht) + \log C'
        , \qquad t \geq 0
        .
    \]
    Pick \(k>0\) according to \Cref{per} with \(h = h'\).
    Let \(l >0\) be arbitrary.
    By \eqref{M12}, there are \(C'',l' >0\) such that
    \[
        \frac{1}{l'}\varphi^*(l't) + (\log 2)t 
        \leq \frac{1}{l}\varphi^*(lt) + \log C''
        , \qquad t \geq 0
        .
    \]
    \Cref{per} implies that there are \(C,s >0\) such that
    \begin{equation}
        \label{per-omega}
        \begin{gathered}
            \forall f \in C^\infty_a(\R^n), \norm{ f }_{\omega,k} \leq 1 \, \forall \varepsilon \in (0,1] \, \exists f_\varepsilon \in C^\infty_a(\R^n) \colon \\ 
            \norm{ f- f_\varepsilon }_{\omega,h'} \leq \varepsilon 
            \quad \text{and} \quad 
            \norm{ f_\varepsilon }_{\omega,l'} \leq \frac{C}{\varepsilon^{s}}
            .
        \end{gathered}
    \end{equation}
    Let \(f \in \mathcal{D}_K\) with \(\norm{ f }_{\omega, k} \leq 1\) be arbitrary.
    Denote by \(f_a\) the \(a\Z^n\)-periodic extension of \(f\) and note that \(\norm{ f_a }_{\omega, k} = \norm{ f }_{\omega, k} \leq 1\).
    Hence, \eqref{per-omega} yields that for all \(\varepsilon \in (0,1]\) there is \(f_{a,\varepsilon} \in C^\infty_a(\R^n)\) such that \(\norm{ f_a- f_{a,\varepsilon} }_{\omega,h'} \leq \varepsilon\) and \(\norm{ f_{a,\varepsilon} }_{\omega,l'} \leq C\varepsilon^{-s}\).
    Let \(\chi \in \mathcal{D}^{(\omega)}_L\) be such that \(\chi =1\) on a neighborhood of \(K\).
    Set \(f_{\varepsilon} = \chi f_{a,\varepsilon} \in \mathcal{D}_L\).
    Then,
    \[
        \norm{ f- f_\varepsilon }_{\omega,h} 
        = \norm{ \chi(f_a-f_{a,\varepsilon})}_{\omega,h} 
        \leq C' \norm{ \chi}_{\omega,h'}\norm{ f_a-f_{a,\varepsilon}}_{\omega,h'} 
        \leq C' \norm{ \chi}_{\omega,h'} \varepsilon
    \]
    and
    \[
        \norm{ f_\varepsilon }_{\omega,l} 
        = \norm{ \chi f_{a,\varepsilon}}_{\omega,l} 
        \leq C'' \norm{ \chi}_{\omega,l'}\norm{ f_{a,\varepsilon}}_{\omega,l'} 
        \leq \frac{CC'' \norm{ \chi}_{\omega,l'}}{\varepsilon^s}
        .
    \]
    The result now follows by rescaling.
\end{proof}
\begin{proof}[Proof of \Cref{main-omega}\ref{main-omega-2}] 
    Fix a left-invariant frame \({\bf{X}}\) on \(G\).
    By \Cref{main-frames-vec-li}\ref{main-frames-vec-li-2}, we have  \(E^{(\omega)} = \bigcap_{h >0} E^{\omega,h}_{\mathbf{X}}\) as locally convex spaces, where we endow the right-hand side with its natural Fr\'{e}chet space topology.
    By \Cref{abstract} and rescaling, it therefore suffices to show that
    \begin{gather*}
        \forall {K \subseteq G \text{ compact},} \ h>0, \ p \in \csn(E) \\
        \exists {L \subseteq G \text{ compact},}\ k >0, \ q \in \csn(E) \\
        \forall {M \subseteq G \text{ compact},}\ l >0, \ r \in \csn(E) \\ 
        \exists C_1,C_2,s > 0 \,
        \forall v \in E^{\omega,k}_{\mathbf{X}},  q_{{\mathbf{X}},{L,}\omega, k}(v) \leq 1 \, \forall \varepsilon \in (0,1] \, \exists v_\varepsilon \in E^{\omega,l}_{\mathbf{X}} \colon \\
        p_{{\mathbf{X}},{K,}\omega, h}(v-v_\varepsilon) \leq C_1\varepsilon 
        \quad \text{and} \quad  
        r_{{\mathbf{X}},{M,}\omega,l}(v_\varepsilon) \leq \frac{C_2}{\varepsilon^s}
        .
    \end{gather*}
  
{
    Let \(K \subseteq G\) compact, \(h >0\), and \(p \in \csn(E)\) be arbitrary.
    Without loss of generality, we may assume that \(K\) is a neighborhood of \(e\).
    Fix a compact subset \(L\) of \(G\) with \(K \subseteq \operatorname{int} L\).
    }
    By \Cref{decomp-G}, there is \(h' >0\) such that
    \begin{equation}
        \label{om-1}
        \begin{gathered}
            \forall l >0 \, \exists C',s' > 0\, \forall f \in \mathcal{D}_K,  \norm{ f }_{\mathbf{X},\omega,h'} \leq 1 \, \forall \varepsilon \in (0,1] \, \exists f_\varepsilon \in \mathcal{D}_L \colon \\ 
            \norm{ f-f_\varepsilon }_{\mathbf{X},\omega,h} \leq \varepsilon 
            \quad \text{and} \quad 
            \norm{ f_\varepsilon }_{\mathbf{X},\omega,l} \leq \frac{C'}{\varepsilon^{s'}}
            .
        \end{gathered}
    \end{equation}
    By \Cref{par-final}\ref{par-final-2}, there are \(k >0\), \(\chi_j \in \mathcal{D}_K\) with \(\norm{\chi_j}_{{\mathbf{X}},\omega,h'} < \infty\), and continuous linear mappings \(T_j\colon E^{\omega,k}_{\mathbf{X}} \to E\), \(j=1, \ldots, 2^n\), such that
    \begin{equation}
        \label{par-use}
        v = \sum_{j = 1}^{2^n} \Pi(\chi_j)T_j(v), \qquad v \in E^{\omega,k}_{\mathbf{X}}.
    \end{equation}
  \Cref{smoothening}\ref{smoothening-2} implies that there is \(p' \in \csn(E)\) such that
    \[
        p_{{\bf{X}}, K,\omega,h}(\Pi(\chi)w) 
        \leq \norm{ \chi }_{{\bf{X}},\omega,h} p'(w)
        , \qquad w \in E, \chi \in \mathcal{D}_L
        .
    \]
    
    Since \(E\) satisfies \((\Omega)\), there is \(q' \in \csn(E)\), \(q' \geq p'\), such that
    \begin{equation}
        \label{om-2}
        \begin{gathered}
            \forall r' \in \operatorname{csn}(E) \, \exists C'', s'' > 0 \, \forall w \in E, q'(w) \leq 1 \, \forall \delta \in (0,1] \, \exists \, w_\delta \in E\colon \\ 
            p'(w-w_\delta) \leq \delta \quad \text{and} \quad r'(w_\delta) \leq \frac{C''}{\delta^{s''}}
            .
        \end{gathered}
    \end{equation}
    
    Pick \(\tilde{q} \in \csn(E)\) and \(\tilde{L} \subseteq G\) compact such that for all \(j = 1, \ldots, 2^n\) it holds that
    \[
        q'(T_j(v)) 
        \leq \tilde{q}_{{\bf{X}},\tilde{L},\omega,k}(v)
        , \qquad v \in E^{\omega,k}_{\mathbf{X}}
        .
    \]
As \(L\) has non-empty interior, there exist \(x_{1},\dots,x_{m} \in G\) such that \(\tilde{L} \subseteq \bigcup_{i=1}^{m} x_{i}L\).
   Since the mappings  \(\pi(x_i): E \to E\), $i =1, \ldots, m$, are continuous, there is \(q \in \csn(\Vecs)\) such that   for all \(j = 1, \ldots, 2^n\)
    \[
       q'(T_j(v)) 
        \leq  \tilde{q}_{{\bf{X}},\tilde{L},\omega,k}(v) \leq q_{{\bf{X}},{L,}\omega,k}(v)
        , \qquad v \in E^{\omega,k}_{\mathbf{X}}.
    \]
    
    Let {\(M \subseteq G\) compact,} \(l>0\), and \(r \in \csn(E)\) be arbitrary.
    We may assume that \(l \leq h\).
    \Cref{smoothening}\ref{smoothening-2} implies that there is \(r' \in \csn(E)\) such that
    \[
        r_{{\bf{X}},M,\omega,l}(\Pi(\chi)w) 
        \leq \norm{ \chi }_{{\bf{X}},\omega,l} r'(w)
        , \qquad w \in E, \chi \in \mathcal{D}_L
        .
    \]
    Choose \(C',s'>0\) and  \(C'',s''>0\) according to  \eqref{om-1} and  \eqref{om-2}, respectively.
    Let \(v \in E^{\omega,k}_{\mathbf{X}}\) with  \(q_{{\mathbf{X}},{L,}\omega,k}(v) \leq 1\) be arbitrary.
    Condition \eqref{om-1} and a rescaling argument imply that there is \(C''' >0\) such that for all \(j = 1, \ldots, 2^n\) and \(\varepsilon \in(0,1]\) there is \(\chi_{j,\varepsilon} \in \mathcal{D}_L\) with  \(\norm{ \chi_j - \chi_{j,\varepsilon} }_{\mathbf{X},\omega,h} \leq \varepsilon\) and \(\norm{ \chi_{j,\varepsilon} }_{\mathbf{X},\omega,l} \leq C'''\varepsilon^{-s'}\).
    Note that \(q'(T_j(v)) \leq q_{{\bf{X}},{L,}\omega,k}(v) \leq 1\) for all \(j = 1, \ldots, 2^n\).
    Hence, by \eqref{om-2}, we obtain that for all \(j = 1, \ldots, 2^n\) and \(\delta \in(0,1]\) there is \(w_{j,\delta} \in E\) with \(p'(T_j(v) - w_{j,\delta}) \leq \delta\) and \(r'(w_{j,\delta}) \leq C''\delta^{-s''}\).
    Equation \eqref{par-use} gives that for all \(\varepsilon, \delta \in (0,1]\)
    \begin{align*}
        &p_{{\mathbf{X}},{K,}\omega,h}(v- \sum_{j = 1}^{2^n} \Pi(\chi_{j,\varepsilon})w_{j,\delta}) \\
        &\qquad= p_{{\mathbf{X}},{K,}\omega,h}(\sum_{j = 1}^{2^n} \Pi(\chi_j)T_j(v)- \Pi(\chi_{j,\varepsilon})w_{j,\delta}) \\
        &\qquad\leq \sum_{j = 1}^{2^n} p_{{\mathbf{X}},{K,}\omega,h}(\Pi(\chi_j- \chi_{j,\varepsilon})T_j(v)) + p_{{\mathbf{X}},{K,}\omega,h}(\Pi(\chi_{j,\varepsilon})(T_j(v) -w_{j,\delta})) \\
        &\qquad\leq \sum_{j = 1}^{2^n} \norm{\chi_j- \chi_{j,\varepsilon} }_{{\mathbf{X}},\omega,h} p'(T_j(v)) + \norm{ \chi_{j,\varepsilon} }_{{\mathbf{X}},\omega,h} p'(T_j(v) -w_{j,\delta}) \\
        &\qquad\leq 2^n \varepsilon + \frac{2^n C'''}{\varepsilon^{s'}} \delta
    \end{align*}
    and
    \[
        r_{{\mathbf{X}},{M,}\omega,l}(\sum_{j = 1}^{2^n} \Pi(\chi_{j,\varepsilon})w_{j,\delta}) 
        \leq \sum_{j = 1}^{2^n}  \norm{ \chi_{j,\varepsilon} }_{{\mathbf{X}},\omega,l} r'(w_{j,\delta}) 
        \leq \frac{2^n C''C'''}{\varepsilon^{s'} \delta^{s''}}
        .
    \]
    We set \(s = s' + (s'+1)s''\) and \(\delta_\varepsilon = \varepsilon^{s'+1}\) for \(\varepsilon \in (0,1]\).
    Then, for all \(\varepsilon \in (0,1]\), we have  \(\sum_{j = 1}^{2^n} \Pi(\chi_{j,\varepsilon})w_{j,\delta_\varepsilon} \in E^{\omega,l}_{\mathbf{X}}\) by \Cref{smoothening}\ref{smoothening-2}.
    Moreover,
    \[
        p_{{\mathbf{X}},{K,}\omega,h}(v- \sum_{j = 1}^{2^n} \Pi(\chi_{j,\varepsilon})w_{j,\delta_\varepsilon}) 
        \leq 2^n(1 + C''')\varepsilon
    \]
    and
    \[
        r_{{\mathbf{X}},{M,}\omega,l}\left( \sum_{j = 1}^{2^n} \Pi(\chi_{j,\varepsilon})w_{j,\delta_\varepsilon}\right) 
        \leq \frac{2^n C''C'''}{\varepsilon^{s}}
        .
    \]
\end{proof}

\begin{remark}
    \label{rmk:extension-action-property-omega}
    As for \Cref{main} (see \Cref{rmk:extension-action-quasinormability}), \Cref{main-omega} in fact holds for any representation \(\pi\) for which \(\Pi \colon C_{c}(G) \times E^{0} \to E^{0}\) extends to a bilinear map \(C_{c}(G) \times E \to E\) such that the statements in \Cref{smoothening} hold for all \(v \in E\).
  The proof works exactly in the same way under this new assumption.
\end{remark}

\begin{remark}
    We do not know whether it is possible to use the same more elementary technique as in the proof of \Cref{main-omega}\ref{main-omega-1} to show \Cref{main-omega}\ref{main-omega-2}.
    The problem is that we are unable to find a family \((\chi_\varepsilon)_{\varepsilon \in (0,1]} \subseteq \mathcal{D}^{(\omega)}(G)\) that satisfies the assumptions of \Cref{approx} and suitable polynomial type bounds in \(\varepsilon\) with respect to the scale of norms \((\norm{ \Cdot}_{{\bf{X}},\omega,h})_{h >0}\) (cf.\ \eqref{boundes-eps}).
\end{remark}

\section{Right-invariant spaces of smooth and ultradifferentiable functions}\label{sect-examplesI}
A lcHs \(E\) is said to be a \emph{right-invariant locally convex Hausdorff function space (= right-invariant lcHfs) on \(G\)} if \(E\) is non-trivial and satisfies the following four properties:
\begin{enumerate}[label=\textup{(A.\arabic*)},series=rifs]
    \item \label{cnd:A1} \(E\) is continuously embedded into \(\Lloc(G)\).
    \item \label{cnd:A2} \(R_x(E) \subseteq E\) for all \(x \in G\).
    \item \label{cnd:A3} For all \(K \subseteq G\) compact and \(p \in \csn(E)\) there is \(q \in \csn(E)\) such that
        \[
            \sup_{x\in K} p( R_x f) 
            \leq q(f)
            , \qquad \forall f \in E
            .
        \]
    \item \label{cnd:A4} \(E \ast C_c(G) \subseteq E\). 
    {\item \label{cnd:A5} The mapping \(E \times C_{c}(G) \to E, \, (f,\chi) \mapsto f \ast \chi\) is separately continuous.}
    {
        \item \label{cnd:A6} For all compact \(K \subseteq G\) and \(p \in \csn(\Vecs)\), there exists \(q \in \csn(\Vecs)\) such that
            \[
                p(f \ast \chi) 
                \leq \sup_{x \in K} \abs{\chi(x)} \, q(f)
                , \qquad \forall f \in \Vecs, \chi \in C(G) \text{ with } \operatorname{supp} \chi \subseteq K 
                .
            \]
    }
     
\end{enumerate}

Conditions \ref{cnd:A1}-\ref{cnd:A3} imply that the right-regular representation
\[
    \pi_{R}\colon G \to \operatorname{GL}(E), \, x \mapsto \pi_{R}(x) \coloneqq R_x
\]
is well-defined and locally equicontinuous.
Condition \ref{cnd:A4} means that \(E\) is a right-module over the (left-)convolution algebra \(C_c(G)\).
{
Condition \ref{cnd:A6} is equivalent to the following statement:
For all \(B \subseteq C_c(G)\) bounded, the family of mappings 
\[
\{ f \mapsto f \ast \chi \colon \Vecs \to \Vecs \mid \chi \in B\}
\] 
is equicontinuous.
This may be considered a local equicontinuity condition for the natural action of the algebra \(C_{c}(G)\) on \(\Vecs\) by convolution and ensures that \(\Vecs\) is a bornological module over \(\CCont(\LieG)\).
A Banach space that is a right-invariant lcHfs on \(G\) is simply called a \emph{right-invariant \((Bf)\)-space on \(G\)}.
Our definition of a right-invariant lcHfs is inspired by the Banach function spaces used in the coorbit theory of Feichtinger and Gröchenig \cite{F-G} as well as the translation-invariant Banach spaces of distributions from \cite{D-P-V}.

\begin{example} 
    Let \(\rho\) be a right-invariant Haar measure on \(G\).
    The Banach  spaces \(L^p(G) = L^p(G,\rho)\), \(1 \leq p \leq \infty\), are right-invariant \((Bf)\)-spaces.
    We define \(L^0(G)\) as the space of all \(f \in L^\infty(G)\) such that for all \(\varepsilon >0\) there is a compact subset \(K\) of \(G\) such that
    \[
        \operatorname*{ess\,sup}_{x\in G \setminus K}\abs{f(x)} 
        \leq \varepsilon
    \]
    and endow \(L^0(G)\) with the subspace topology induced by \(L^\infty(G)\).
    Then, \(L^0(G)\) is a right-invariant \((Bf)\)-space as well.
    Weighted variants of the spaces \(L^p(G)\), \(p \in \{0\} \cup [1,\infty]\) (cf.\ the introduction) will be discussed in the next section.
\end{example}

\begin{remark}
    If \(E\) is webbed and ultrabornological, and \ref{cnd:A1} and \ref{cnd:A2} hold, then \ref{cnd:A3} is automatically satisfied, as follows from the closed graph theorem of De Wilde \cite[Theorem 24.31]{M-V} and the Banach-Steinhaus theorem.
 If \(\Vecs\) is barrelled, \ref{cnd:A6} follows  from \ref{cnd:A5} by an application of the Banach-Steinhaus theorem.
\end{remark}

\begin{lemma}\label{conv-repr}
    Let \(E\) be a sequentially complete lcHs satisfying \ref{cnd:A1}-{\ref{cnd:A4}} and consider the right-regular representation \(\pi_{R}\) of \(G\) on \(E\).
    Then,
    \[
        \Pi_{R}(\chi) f 
        = f \ast \chi
        , \qquad f \in E^0, \chi \in C_c(G)
        .
    \]
\end{lemma}
\begin{proof}
    
    For the sake of clarity, we will momentarily explicitly denote the continuous embedding \(\Vecs \subseteq \Lloc(\LieG)\) from \ref{cnd:A1} as \(\iota:\Vecs \to \Lloc(\LieG)\).
  By definition, \(\iota (\RMul{x}f) = \RMul{x}\iota (f) \) and \(\iota (f \ast \chi )= \iota(f) \ast \chi \), for all \(x \in \LieG\), \(f \in \Vecs\), and \(\chi \in C_{c}(\LieG)\).
    Now, take arbitrary \(f \in E^{0}\) and \(\chi \in C_{c}(\LieG)\).
    We have
    \[
        \iota(\Pi_{R}(\chi)f)
        = \int_{\LieG} \iota(\chi(y^{-1}) \RMul{y}f) \dl{y}
        = \int_{\LieG} \chi(y^{-1}) \RMul{y}\iota(f) \dl{y}
        = \iota(f) \ast \chi
        = \iota(f \ast \chi)
        .
    \]

\end{proof}

\begin{remark} 
    Let \(E\) be a sequentially complete lcHs satisfying \ref{cnd:A1}-{\ref{cnd:A4}}.
    Suppose that the right-regular representation of \(G\) on \(E\) is continuous (this particularly holds true if \(C_c(G)\), or more generally \(\mathcal{D}^{(\sigma)}(G)\) for some weight function \(\sigma\), is continuously and densely embedded in \(E\)).
   \Cref{smoothening}\ref{smoothening-1} and \Cref{conv-repr} yield that \(E\) satisfies {\ref{cnd:A6} -- hence also \ref{cnd:A5} --} and therefore is a right-invariant lcHfs on \(G\).
    However, there are right-invariant lcHfs \(E\) such that the right-regular representation of \(G\) on \(E\) is not continuous, consider for example
    \(E =L^\infty(G)\).
\end{remark}

Let \(E\) be a sequentially complete right-invariant lcHfs on \(G\).
Our goal is to  determine the spaces \(E^\infty\) and \(E^{(\omega)}\) for the right-regular representation of \(G\) on \(E\).
Let \(\mathbf{X}\) be a left-invariant frame on \(G\), let \(f \in C^\infty(G)\), let \(j \in \N\) and \(\alpha \in \{1,\ldots,n\}^j\), \(j \in \N\).
Recall \(\pi_{R}(X^{\alpha})\) is the operator acting on \(E^{{j}}\) via \eqref{eq:infinitesimalaction}, while \(X^\alpha\) acts on elements of \(C^{{j}}(G)\). 

\begin{example}
\label{exa:regular-vectors-right-regular-representation-continuous-compact}

The space $C(G)$ is a right-invariant lcHfs. Using \eqref{equ:associated-left-invariant-vector-field}, one may readily verify that, for $j \in \N$, $C^j(G) = (C(G))^j$ and that the operators \(\RRep(\VectorField^{\alpha})\) and \(\VectorField^{\alpha}\) coincide, where $\alpha \in \{1,\dots, \Dim\}^j$.
\end{example}

\begin{lemma}
    \label{thm:derivative-convolution-function-vector} 
    Let \(\Frame\) be a left-invariant frame on \(\LieG\) and let \(\Vecs\) be a right-invariant lcHfs on $G$.
    Then,
    \[
        \RRep(\VectorField^\alpha) (f \ast \chi)
        = f \ast (\VectorField^\alpha \chi)
        , \qquad j \in \N, \alpha \in \{1,\dots, \Dim\}^j, f \in \Vecs, \chi \in \CSmooth[j](\LieG)
        .
    \]
\end{lemma}
\begin{proof}
    Suppose \(X \in \LieA\) and \(\chi \in \CSmooth[1](\LieG)\).
    Note that \(\RMul{x}(f \ast \chi) = f \ast (\RMul{x}\chi)\), \(x \in \LieG\), means that \(\Orbit{f \ast \chi}(x) = f \ast(\Orbit{\chi}(x))\), where the second orbit is to be interpreted using the right regular representation on {\(C(\LieG)\)}.
    \Cref{exa:regular-vectors-right-regular-representation-continuous-compact}, \eqref{orbit-prop}, and \eqref{equ:associated-left-invariant-vector-field} then yield
    \[
        \Orbit{\VectorField\chi}(x)
        = \Orbit{\RRep(\VectorField)\chi}(x)
        = \VectorField\Orbit{\chi}(x)
        =  \left. \frac{d}{ dt} \right|_{t =0}\Orbit{\chi}(x\exp t\VectorField)
        , \qquad x \in \LieG
        .
    \]
    Therefore, using \eqref{equ:associated-left-invariant-vector-field} and \ref{cnd:A5}, we find
    \begin{alignat*}{2}
        \VectorField\Orbit{f \ast \chi}(x)
        &= \left. \frac{d}{ dt}  \right|_{t =0}\Orbit{f \ast \chi}(x\exp t\VectorField) 
        =\left.\frac{d}{ dt}   \right|_{t =0}({f \ast (\Orbit{\chi}(x\exp t\VectorField))})
        \\
        & = f \ast \left(\left. \frac{d}{ dt}  \right|_{t =0}\Orbit{\chi}(x\exp t\VectorField)\right) 
       = f \ast \left(\Orbit{\VectorField\chi}(x)\right) 
        = \Orbit{f \ast \left(\VectorField\chi\right)}(x)
        ,
    \end{alignat*}
    for all \(x \in \LieG\).
    By induction, we get
    \[
        \VectorField^{\alpha}\Orbit{f \ast \chi}
        = \Orbit{f \ast \left(\VectorField^{\alpha}\chi\right)}
        ,
    \]
    for all \(j \in \N\), \(\alpha \in \{1,\dots, \Dim\}^j\), \(f \in \Vecs\), and \(\chi \in \CSmooth[j](\LieG)\).
    The statement then follows by evaluating at \(\Origin\).
\end{proof}}

The previous result particularly implies
\begin{equation}
    \label{equ:smoothening-convolution-rifs}
    \Vecs \ast \CSmooth[j](\LieG) \subseteq \SmoothV[j]
    , \qquad j \in \N \cup \{\infty\}
    .
\end{equation}
We can make this lifting of regularity more precise.

\begin{lemma}
    \label{thm:boundedness-convolution}
   Let \(\Frame\) be a left-invariant frame on $G$ and  let \(\Vecs\) be a right-invariant lcHfs on $G$. Let \(K,L \subseteq \LieG\) be compact and \(p \in \csn(\Vecs)\).
      \begin{enumerate}[i]
        \item 
            There is \(q \in \csn(\Vecs)\) such that for all \(j \in \N\) 
            \[
                p_{{\bf{X}},K,j}(f\ast \chi) 
                \leq \norm{ \chi }_{{\bf{X}},j} q(f) 
                , \qquad \forall f \in E, \chi \in C^j_c(G) \text{ with } \operatorname{supp} \chi \subseteq L
                .
            \]
        \item 
            There is \(q \in \csn(\Vecs)\) such that for all \(h >0\) 
            \[
                p_{{\bf{X}},K,\omega,h}(f\ast \chi) 
                \leq  \norm{ \chi }_{{\bf{X}},\omega,h} q(v)
                , \qquad \forall f \in E, \chi \in \mathcal{D}_L
                .
            \]    
    \end{enumerate}
\end{lemma}

\begin{proof}
    We only show (i) as (ii) is similar. 
    Take \(q\) as in \ref{cnd:A3} and \(r\) as in \ref{cnd:A6} for \(p=q\) and \(K=L\).
    By \Cref{thm:derivative-convolution-function-vector} and \eqref{orbit-prop}, we have for all \(f \in \Vecs\) and \(\chi \in \CSmooth[j](\LieG)\), $j \in \N$, with \(\Supp \chi \subseteq L\)
    \begin{align*}
        p_{\Frame,K,j}(f \ast \chi)
        &= \max_{i \leq j} \max_{\alpha \in \{1,\dots,\Dim\}^{i}} \sup_{x \in K} p\left(\VectorField^{\alpha}\Orbit{f \ast \chi}(x)\right) \\
        &\leq \max_{i \leq j} \max_{\alpha \in \{1,\dots,\Dim\}^{i}} q\left(\RRep(\VectorField^{\alpha})(f \ast \chi)\right) \\
        &= \max_{i \leq j} \max_{\alpha \in \{1,\dots,\Dim\}^{i}} q\left(f \ast (\VectorField^{\alpha}\chi)\right) \\
        &\leq \max_{i \leq j} \max_{\alpha \in \{1,\dots,\Dim\}^{i}} \sup_{x \in L} \abs{\VectorField^{\alpha}\chi(x)} \, r(f) \\
        &= \norm*{\chi}_{\Frame,j} \, r(f) 
    \end{align*}
\end{proof}
\begin{remark}
    \label{rmk:extension-action-rifs}
    Suppose \(\Vecs\) is a sequentially complete lcHfs on $G$.
    Using \Cref{conv-repr}, \(\Pi\) can be extended to act on all \(f \in \Vecs\) by means of the map \((\chi,f) \mapsto f \ast \chi\).
    \Cref{thm:boundedness-convolution} guarantees that this extension satisfies the statements from \Cref{smoothening} for all elements in \(\Vecs\).
\end{remark}

The next lemma shows that  \(\pi_{R}({X}^{\alpha}) f\) coincides with \(X^{\alpha} f\) for \(f \in E^{\infty}\cap C^{\infty}(G)\).
In fact, \Cref{description-lemma} below tells us that \(E^{\infty}\subseteq C^{\infty}(G)\), so that the operator \( {\pi_{R}(X^{\alpha})} \colon E^{\infty}\to E^{\infty}\) just becomes the classical derivative {with respect to the left-invariant differential operator \(X^{\alpha}\)}{, cf.\ \Cref{exa:regular-vectors-right-regular-representation-continuous-compact}}.

\begin{lemma}\label{smoothl1loc} 
    Let \(\mathbf{X}\) be a left-invariant frame on \(G\).
    Consider the right-regular representation $\pi_{R}$ of \(G\) on \(\Lloc(G)\).
    Then, \(C^\infty(G) \subseteq \Lloc(G)^\infty\).
    Moreover, one has {\(\pi_{R}(X^\alpha) f =X^\alpha f\) for all \(f \in C^\infty(G)\) and \(\alpha \in \{1,\ldots,n\}^j, j \in \N\).}
\end{lemma}
\begin{proof}
    The dual of \(\Lloc(G)\) may be identified with the space \(L^\infty_c(G)\) consisting of all elements in \(L^\infty(G)\) whose essential support is compact.
    Since \(\Lloc(G)\) is sequentially complete, \Cref{weaksmooth} implies that \((\Lloc(G))^\infty\) consists precisely of all those \(f \in \Lloc(G)\) such that \(f \ast \check{\chi} \in C^\infty(G)\) for all \(\chi \in L^\infty_c(G)\).
    Consequently, \(C^\infty(G) \subseteq \Lloc(G)^\infty\).
    Let \(f \in C^\infty(G)\) be arbitrary.
    Formula \eqref{equ:associated-left-invariant-vector-field} implies that \(X_j \gamma_f = \gamma_{{X_{j}} f}\) for all \(j = 1, \ldots,n\).
    Hence, by using induction, we find that \(X^\alpha \gamma_f = \gamma_{{X^{\alpha}} f}\) for all \(\alpha \in \{1,\ldots,n\}^j\), \(j \in \N\).
  Evaluating this equality at \(e\), we obtain that {\(\pi_{R}(X^{\alpha})f = X^{\alpha}f\)} for all \(\alpha \in \{1,\ldots,n\}^j\), \(j \in \N\).
\end{proof}

\begin{remark}
 Lemma \ref{description-lemma} below implies that we actually have \(C^\infty(G) = \Lloc(G)^\infty\). 
\end{remark}

In view of \Cref{conv-repr,smoothl1loc}, \Cref{par-final}\ref{par-final-1} yields the following result.

\begin{lemma}\label{par-smooth} 
    Let \({\bf{X}} = (X_1, \ldots, X_n)\) be a left-invariant frame on \(G\).
    For all \(j \in \N\) there are a polynomial \(P\) and \(\chi_\theta \in C^j_c(G)\), \(\theta = (\theta_1, \ldots, \theta_n) \in \{0,1\}^n\), such that for all \(f \in C^\infty(G)\)
    \[
        f
        = \sum_{\theta \in \{0,1\}^n} (P_{\theta_n}(X_n) \cdots P_{\theta_1}(X_1)f ) \ast \chi_\theta
        ,
    \]
    where \(P_0 = P\) and \(P_1 =1\).
\end{lemma}

Let \(E\) be a sequentially complete right-invariant lcHfs on \(G\) and let \({\bf{X}}\) be a left-invariant frame on \(G\).
We denote by \(\mathcal{D}_{E, {\mathbf{X}}}(G)\) the space consisting of all \(f \in C^\infty(G)\) such that \(X^\alpha f \in E\) for all \(\alpha \in \{1,\ldots,n\}^j\), \(j \in \N\).
For \(j \in \N\) and \(p \in \csn(E)\), we set
\[
    p_{\mathbf{X},j}(f) 
    = \max_{i \leq j} \max_{\alpha \in \{1, \ldots, n\}^i} p(X^\alpha f) 
    , \qquad f \in \mathcal{D}_{E,\mathbf{X}}(G)
    .
\]
We endow \(\mathcal{D}_{E, {\mathbf{X}}}(G)\) with the locally convex topology generated by the system of seminorms \(\{ p_{\mathbf{X},j} \mid j \in \N, p \in \csn(E)\}\).
Likewise,
we write \(\mathcal{D}^{(\omega)}_{E, {\mathbf{X}}}(G)\) for the space consisting of all \(f \in \mathcal{D}_{E, {\mathbf{X}}}(G)\) such that, for all \(h >0\) and \(p \in \csn(E)\),
\[
    p_{\mathbf{X},\omega,h}(f) 
    = \sup_{j \in \N} \max_{\alpha \in \{1, \ldots, n\}^j} p(X^\alpha f)\exp \left(-\frac{1}{h}\varphi^*(hj)\right) 
    < \infty
\]
and endow \(\mathcal{D}^{(\omega)}_{E, {\mathbf{X}}}(G)\) with the locally convex topology generated by the system of seminorms \(\{ p_{\mathbf{X},\omega,h} \mid h>0, p \in \csn(E)\}\).

\begin{remark}
    \label{rmk:non-trivial}
    Proposition \ref{main-frames} and {\Cref{thm:boundedness-convolution}} imply that \(f \ast \chi \in \mathcal{D}_{E, {\mathbf{X}}}(G)\) (\(\mathcal{D}^{(\omega)}_{E, {\mathbf{X}}}(G)\)) for all \(f \in E\) and \(\chi \in \mathcal{D}(G)\) (\(\chi \in \mathcal{D}^{(\omega)}(G)\)).
    In particular, the spaces \(\mathcal{D}_{E, {\mathbf{X}}}(G)\) and \(\mathcal{D}^{(\omega)}_{E, {\mathbf{X}}}(G)\) are non-trivial (recall that \(E\) is assumed to be non-trivial).
\end{remark}

\begin{lemma}\label{description-lemma}
    Let \(\mathbf{X}\) be a left-invariant frame on \(G\) and let \(E\) be a sequentially complete right-invariant lcHfs on \(G\).
    Then, \(E^\infty = \mathcal{D}_{E, {\mathbf{X}}}(G)\) as sets.
\end{lemma}
\begin{proof}
    We first prove that \(E^\infty \subseteq \mathcal{D}_{E, {\mathbf{X}}}(G)\).
    Let \(f \in E^\infty\) be arbitrary.
    By \Cref{par-final}\ref{par-final-1} and \Cref{conv-repr} it holds that for all \(j \in \N\) there are \(\chi_i \in C^{j}_c(G)\) and \(f_i \in E^\infty \subseteq E^0\), \(i = 1, \ldots, 2^n\), such that
    \[
        f 
        = \sum_{i = 1}^{2^n} \Pi_{R}(\chi_i) f_i 
        = \sum_{i = 1}^{2^n} f_i \ast \chi_i \in C^{j}(G)
        .
    \]
    In particular, \(f \in C^{\infty}(\LieG)\).
    Moreover, for all \(\alpha \in \{1,\ldots,n\}^j\), we have 
    \[
        X^\alpha f 
        = \sum_{i = i}^{2^n} f_i \ast X^\alpha \chi_i 
        \in E
        ,
    \]
    where the last inclusion follows from the fact that \(E \ast C_c(G) \subseteq E\) (condition \ref{cnd:A4}).
    Hence, \(f \in \mathcal{D}_{E, {\mathbf{X}}}(G)\).
    Next we show that \(\mathcal{D}_{E, {\mathbf{X}}}(G) \subseteq E^\infty\).
        Let \(f \in \mathcal{D}_{\Vecs, \Frame}(\LieG)\) and \(j \in \N\) be arbitrary.
    By \cref{par-smooth} there are \(\chi_i \in \CSmooth[j](\LieG)\) and \(f_i \in \Vecs\), \(i = 1, \ldots, 2^{\Dim}\), such that
    \[
        f 
        = \sum_{i = 1}^{2^{\Dim}} f_i \ast \chi_i
        \in \SmoothV[j]
        .
    \]
    Here, the last inclusion follows from \eqref{equ:smoothening-convolution-rifs}.
    Hence, \(f \in E^\infty\). 
\end{proof}

\begin{remark}\label{non-trivial}
    \Cref{description-lemma} for the particular case \(E = L^p(G)\), \(1 \leq p < \infty\), was shown by Poulsen \cite{Poulsen} via  different methods.
\end{remark}

 \Cref{description-lemma} and \Cref{main-frames-vec-li} yield the following result.
\begin{theorem}\label{description}
    Let \(E\) be a sequentially complete right-invariant lcHfs on \(G\) and let \(\mathbf{X}\) be a left-invariant frame on \(G\).
    \begin{enumerate}[i]
        \item \label{description-1}
            \(E^\infty = \mathcal{D}_{E, {\mathbf{X}}}(G)\) as locally convex spaces.
        \item \label{description-2}
            \(E^{(\omega)} = \mathcal{D}^{(\omega)}_{E, {\mathbf{X}}}(G)\) as locally convex spaces.
    \end{enumerate}
\end{theorem}

 Let \(E\) be a sequentially complete right-invariant lcHfs on \(G\).
We write \(\mathcal{D}_{E}(G) = E^\infty\) and \(\mathcal{D}^{(\omega)}_{E}(G) = E^{(\omega)}\). We can thus rephrase \Cref{description} as  \(\mathcal{D}_{E}(G)= \mathcal{D}_{E, {\mathbf{X}}}(G)\) and \(\mathcal{D}^{(\omega)}_{E}(G)= \mathcal{D}^{(\omega)}_{E, {\mathbf{X}}}(G)\) for any left-invariant frame \(\mathbf{X}\) on \(G\).
\Cref{main,main-omega}, and Remarks \ref{rmk:extension-action-quasinormability}, \ref{rmk:extension-action-property-omega}, and \ref{rmk:extension-action-rifs} imply the following two results.

\begin{theorem} \label{main-example}
    Let \(E\) be a sequentially complete right-invariant lcHfs on \(G\).
    \begin{enumerate}[i]
        \item \(\mathcal{D}_{E}(G)\) is quasinormable if \(E\) is so.
        \item \(\mathcal{D}^{(\omega)}_{E}(G)\) is quasinormable if \(E\) is so.
    \end{enumerate}
\end{theorem}

\begin{theorem} \label{main-omega-example}
    Let \(E\) be a Fr\'{e}chet space that is a right-invariant lcHfs on \(G\).
    \begin{enumerate}[i]
        \item \(\mathcal{D}_{E}(G)\) satisfies \((\Omega)\) if \(E\) does so.
        \item \(\mathcal{D}^{(\omega)}_{E}(G)\) satisfies \((\Omega)\) if \(E\) does so.
    \end{enumerate}
\end{theorem}

\section{Examples of right-invariant
Fr\'{e}chet spaces of smooth and ultradifferentiable functions}\label{sect-examplesII}
In this final section, we discuss the quasinormability and the property \((\Omega)\) for a particular class of weighted Fr\'{e}chet spaces of smooth and ultradifferentiable functions.

A right-invariant \((Bf)\)-space \(E\) on \(G\) is said to be \emph{solid} \cite{F1987,F-G} if for all \(f \in E\) and \(g \in \Lloc(G)\) 
\[
    \abs{g(x)} \leq \abs{f(x)} \text{ for almost all \(x \in G\) } \, 
    \Longrightarrow g \in E \text{ and } \norm{g}_E \leq \norm{f}_E
    .
\]
We will also use the following assumption on a right-invariant \((Bf)\)-space \(E\) (cf.\ \cite{F1987}):
\begin{enumerate}[resume*=rifs]
    \item \label{cnd:A7} \(L_x(E) \subseteq E\) for all \(x \in G\) and there is \(B > 0\) such that
        \[
            \norm{ L_xf}_E 
            \leq B\norm{ f}_E
            , \qquad \forall f \in E, x \in G
            .
        \]
\end{enumerate}

\begin{example}\label{intro-example} 
    The spaces \(L^p(G)\), \(p \in \{0\} \cup [1,\infty]\), are solid right-invariant \((Bf)\)-spaces.
    The spaces \(L^0(G)\) and \(L^\infty(G)\) satisfy \ref{cnd:A7}, whereas \(L^p(G)\), \(1 \leq p < \infty\), satisfy \ref{cnd:A7} if \(G\) is unimodular.
\end{example}

By a \emph{weight function system on \(G\)} we mean a pointwise non-decreasing sequence \(\mathcal{V} = (v_j)_{j \in \N}\) of strictly positive continuous functions on \(G\) satisfying the condition \eqref{wfs} from the introduction.
Given a solid right-invariant \((Bf)\)-space \(E\) on \(G\), we denote by \(E_{\mathcal{V}}\) the space consisting of all \(f \in \Lloc(G)\) such that \(fv_j \in E\) for all \(j \in \N\).
We set
\[
    \norm{ f }_{E,v_j} 
    = \norm{ f v_j}_E
    , \qquad f \in E_{\mathcal{V}},\ j \in \N, 
\]
and endow \(E_{\mathcal{V}}\) with the locally convex topology generated by the norms \(\{ \norm{ \Cdot }_{E,v_j} \mid j \in \N\}\).
Then, \(E_{\mathcal{V}}\) is a Fr\'{e}chet space that is a right-invariant lcHfs on \(G\).

\begin{example}\label{example-df}
    Suppose that \(G\) is connected.
    Fix a left- or right-invariant Riemannian metric on \(G\) and consider the associated distance function \(d\colon G \times G \to [0,\infty)\) \cite{BK,Garding}.
    Set \(d(x) = d(e,x)\) for \(x \in G\).
    Then, \(d\colon G \to [0,\infty)\) is continuous and subadditive, i.e., \(d(xy) \leq d(x) + d(y)\) for all \(x,y \in G\).
    Let \((w_j)_{j \in \N}\) be a pointwise non-decreasing sequence of strictly positive continuous functions on \([0,\infty)\) such that
    \[
        \forall i \in \N \, \exists j \in \N, C >0 \, \forall t \geq 0 \colon  
        w_i(t+1) \leq C w_j(t)
        .
    \]
    Then, \((w_j \circ d)_{j \in \N}\) is a weight function system on \(G\).
\end{example}

Let \(E\) be a solid right-invariant \((Bf)\)-space on \(G\) and let \(\mathcal{V}\) be a weight function system on \(G\).
In the next two results, we characterize the quasinormability and the property \((\Omega)\) for \(\mathcal{D}_{E_{\mathcal{V}}}(G)\) and \(\mathcal{D}^{(\omega)}_{E_{\mathcal{V}}}(G)\) in terms of \(\mathcal{V}\).
\begin{theorem} \label{main-t-example1}
    Let \(E\) be a solid right-invariant \((Bf)\)-space on \(G\) and let \(\mathcal{V} = (v_j)_{j \in \N}\) be a weight function system on \(G\).
    Consider the following statements:
    \begin{enumerate}[i]
        \item \label{main-t-example1-1}
            \(\mathcal{V}\) satisfies the condition
            \begin{equation}
                \label{Q-inv}
                \forall i \in \N \, \exists j \geq i \, \forall m \geq j \, \forall \varepsilon \in (0,1] \, \exists C >0 \, \forall x \in G \colon \frac{1}{v_j(x)} \leq \frac{\varepsilon }{v_i(x)} + \frac{C}{v_m(x)}.
            \end{equation}
        \item \label{main-t-example1-2}
            \(E_{\mathcal{V}}\) is quasinormable.
        \item \label{main-t-example1-3}
            \(\mathcal{D}_{E_{\mathcal{V}}}(G)\) (\(\mathcal{D}^{(\omega)}_{E_{\mathcal{V}}}(G)\)) is quasinormable.
    \end{enumerate}
    Then, \ref{main-t-example1-1} \(\Rightarrow\) \ref{main-t-example1-2} \(\Rightarrow\) \ref{main-t-example1-3}.
    If in addition \(E\) satisfies \ref{cnd:A7}, then \ref{main-t-example1-3} \(\Rightarrow\) \ref{main-t-example1-1}.
\end{theorem}
\begin{proof}
    \ref{main-t-example1-1} \(\Rightarrow\) \ref{main-t-example1-2}
    By \Cref{QN-F} it suffices to show that
    \[
        \forall i \in \N \, \exists j \geq i \, \forall \, m \geq j \, \forall \varepsilon \in (0,1] \, \exists C> 0 \colon 
        V_{\norm{ \Cdot }_{E,v_j}} \subseteq \varepsilon V_{\norm{ \Cdot }_{E,v_i}} + C V_{\norm{ \Cdot }_{E,v_m}}
        .
    \]
    Let \(i \in \N\) be arbitrary.
    Choose \(j \geq i\) according to \eqref{Q-inv}.
    Let \(m \geq j\) and \(\varepsilon \in (0,1]\) be arbitrary.
    By \eqref{Q-inv} there is \(C > 0\) such that
    \[
        \frac{1}{v_j(x)} 
        \leq \frac{\varepsilon }{2v_i(x)} + \frac{C}{v_m(x)}
        , \qquad \forall x \in G
        .
    \]
    Note that
    \begin{equation}
        \label{implication}
        v_i(x) \geq \varepsilon v_j(x) \, \Longrightarrow \, v_m(x) \leq 2Cv_j(x), \qquad x \in G.
    \end{equation}
    Let \(\chi_\varepsilon\) be the indicator function of the set
    \[
        \{ x \in G \mid v_i(x) \leq \varepsilon v_j(x) \}
        .
    \]
    Let \(f \in V_{\norm{ \Cdot }_{E,v_j}}\) be arbitrary.
    Since \(E\) is solid, we have \(f\chi_\varepsilon \in E_{\mathcal{V}}\).
    For all \(x \in G\)
    it holds that
    \[
        \abs{f(x) \chi_\varepsilon(x)}v_i(x) 
        \leq \varepsilon \abs{f(x)} v_j(x)
    \]
    and, by \eqref{implication},
    \[
        \abs{f(x)(1- \chi_\varepsilon(x))}v_m(x) 
        \leq 2C\abs{f(x)} v_{j}(x)
        .
    \]
Using that \(E\) is solid and \(\norm{ f }_{E,v_j} \leq 1\), we obtain that \(\norm{f \chi_\varepsilon}_{E,v_i} \leq \varepsilon\) and \(\norm{f(1-\chi_\varepsilon)}_{E,v_m} \leq 2C\), whence
    \[
        f 
        = f \chi_\varepsilon + f(1-\chi_\varepsilon) 
        \in \varepsilon V_{\norm{ \Cdot }_{E,v_i}} + 2C V_{\norm{ \Cdot }_{E,v_m}}
        .
    \]
    \ref{main-t-example1-2} \(\Rightarrow\) \ref{main-t-example1-3}
    This follows from \Cref{main-example}.
    \\
    \ref{main-t-example1-3} \(\Rightarrow\) \ref{main-t-example1-1}
    \emph{(under the extra assumption \ref{cnd:A7})} 
    We will show that \(\mathcal{V}\) satisfies \eqref{Q-inv} if \(\mathcal{D}^{(\omega)}_{E_{\mathcal{V}}}(G)\) is quasinormable (the proof that \(\mathcal{V}\) satisfies \eqref{Q-inv} if \(\mathcal{D}_{E_{\mathcal{V}}}(G)\) is quasinormable is very similar).
    Let \(\mathbf{X}\) be a left-invariant frame on \(G\).
    For \(h >0\) and \(j \in \N\) we write
    \[
        \norm{ f }_{j,h} 
        = \sup_{l \in \N} \max_{\alpha \in \{1, \ldots, n\}^l} \norm{X^\alpha f}_{E,v_j}\exp \left(-\frac{1}{h}\varphi^*(hl)\right)
        , \qquad f \in \mathcal{D}^{(\omega)}_{E_{\mathcal{V}}}(G)
        .
    \]
    As \(\mathcal{D}^{(\omega)}_{E_{\mathcal{V}}}(G)\) is quasinormable, \Cref{QN-F} and \Cref{description}\ref{description-2} imply that
    \begin{gather}
        \label{QND}
        \forall i' \in \N \, \exists j' \geq i', h >0 \, \forall \, m' \geq j' \, \forall \varepsilon \in (0,1] \, \exists C> 0 \\ \nonumber
        \forall f \in \mathcal{D}^{(\omega)}_{E_{\mathcal{V}}}(G), \norm{ f }_{j',h} \leq 1\, \exists f_1,f_2 \in E_{\mathcal{V}}\colon f = f_1+f_2, \, \\ \nonumber
        \norm{ f_1 }_{E, v_{i'}} \leq \varepsilon \quad \text{and} \quad \norm{f_2 }_{E, v_{m'}} \leq C
        .
    \end{gather}
    Let \(f \in \mathcal{D}_{E_{\mathcal{V} }}^{(\omega)}(G) \backslash \{0\}\) (\Cref{rmk:non-trivial}).
    Choose \(\chi \in \mathcal{D}^{(\omega)}(G)\) such that \(f\chi \neq 0\).
    Set \(K = \operatorname{supp} \chi\).
    Condition \eqref{M12} and the fact that \(E\) is solid imply that \(g = f \chi \in \mathcal{D}_{E_{\mathcal{V}}}^{(\omega)}(G)\).
    Let \(B>0\) be as in condition \ref{cnd:A7}.
    Let \(i \in \N\) be arbitrary.
    By \eqref{wfs} there are \(i' \geq i\) and \(C' > 0\) such that
    \begin{equation}
        \label{ineq2}
        v_i(xy^{-1}) \leq C'v_{i'}(x), \qquad x \in G, y \in K.
    \end{equation}
    Choose \(j' \geq i'\) and \(h>0\) according to \eqref{QND}.
    Condition \eqref{wfs} yields that there are \(j \geq j'\) and \(C'' > 0\) such that
    \begin{equation}
        \label{ineq1}
        v_{j'}(xy) \leq C''v_{j}(x), \qquad x \in G, y \in K.
    \end{equation}
    Let \(m \geq j\) be arbitrary.
    Another application of \eqref{wfs} gives us that there are \(m' \geq m\) and
    \(C''' > 0\) such that
    \begin{equation}
        \label{ineq3}
        v_m(xy^{-1}) \leq C'''v_{m'}(x), \qquad x \in G, y \in K.
    \end{equation}
    Let \(\varepsilon >0\) be arbitrary and choose \(C >0\) according to \eqref{QND}.
    Let \(x \in G\) be arbitrary.
    The left-invariance of \(\mathbf{X}\), the fact that \(E\) is solid, condition \ref{cnd:A7}, and \eqref{wfs} imply that \(L_x g \in \mathcal{D}_{E_{\mathcal{V}}}^{(\omega)}(G)\) and
    \[
        \norm{ L_x g }_{j',h} 
        \leq B \sup_{y\in K} v_{j'}(xy) \norm{ g }_{j',h}
        .
    \]
    Hence, by \eqref{ineq1},
    \[
        \norm{ L_x g }_{j',h} 
        \leq BC'' \norm{ g }_{j',h} v_j(x)
        .
    \]
    Set \(C_0 =  BC'' \norm{ g }_{j',h}\).
    Applying \eqref{QND} to \(f = L_xg / ( C_0 v_j(x))\), we find that there are
    \(g_1,g_2 \in E_{\mathcal{V}}\) with \(\norm{ g_1}_{E, v_{i'}} \leq \varepsilon\) and \(\norm{ g_2 }_{E, v_{m'}} \leq C\) such that
    \[
        \frac{L_x g}{C_0v_j(x)} 
        = g_1 + g_2
        .
    \]
    Since \(\operatorname{supp} g \subseteq K\), we have 
    \[
        \frac{g}{C_0v_j(x)} 
        = (L_{x^{-1}}g_1) 1_K + (L_{x^{-1}}g_2) 1_K
        ,
    \]
    where \(1_K\) is the indicator function of \(K\).
    The fact that \(E\) is solid, condition \ref{cnd:A7}, and \eqref{ineq2} imply that \((L_{x^{-1}}g_1) 1_K \in E\) and
    \[
        \norm{ (L_{x^{-1}}g_1) 1_K }_E 
        = \norm{ L_{x^{-1}}(g_1v_{i'}) L_{x^{-1}}\left(\frac{1}{v_{i'}}\right) 1_K }_E 
        \leq B\norm{g_1}_{E,v_{i'}} \sup_{y \in K} \frac{1}{v_{i'}(xy)} 
        \leq \frac{BC' \varepsilon}{v_i(x)}
        .
    \]
    Similarly, by using \eqref{ineq3} instead of \eqref{ineq2}, we have \((L_{x^{-1}}g_2) 1_K \in E\) and
    \[
        \norm{ (L_{x^{-1}}g_2) 1_K }_E 
        \leq \frac{BCC'''}{v_m(x)}
        .
    \]
    Hence,
    \[
        \frac{\norm{g}_E}{C_0v_j(x)} 
        \leq \norm{(L_{x^{-1}}g_1) 1_K }_E + \norm{(L_{x^{-1}}g_2) 1_K}_{E}
        \leq \frac{BC' \varepsilon}{v_i(x)} + \frac{BCC'''}{v_m(x)}
        .
    \]
    The result now follows from a simple rescaling argument (note that \(\norm{g}_E >0\) as \(g \neq 0\)).
\end{proof}

\begin{theorem} \label{main-t-example2}
    Let \(E\) be a solid right-invariant \((Bf)\)-space on \(G\) and let \(\mathcal{V} = (v_j)_{j \in \N}\) be a weight function system on \(G\).
    Consider the following statements:
    \begin{enumerate}[i]
        \item \label{main-t-example2-1}
            \(\mathcal{V}\) satisfies the condition
            \begin{equation}
                \label{omegaw}
                \forall i \in \N \, \exists j \geq i \, \forall m \geq j \, \exists C,s >0 \, \forall \varepsilon \in (0,1] \, \forall x \in G \colon \frac{1}{v_j(x)} \leq \frac{\varepsilon}{v_i(x)} + \frac{C}{\varepsilon^s}\frac{1}{v_m(x)}.
            \end{equation}
        \item \label{main-t-example2-2} 
            \(E_{\mathcal{V}}\) satisfies \((\Omega)\).
        \item  \label{main-t-example2-3}
            \(\mathcal{D}_{E_{\mathcal{V}}}(G)\) (\(\mathcal{D}^{(\omega)}_{E_{\mathcal{V}}}(G)\)) satisfies \((\Omega)\).
    \end{enumerate}
    Then, \ref{main-t-example2-1} \(\Rightarrow\) \ref{main-t-example2-2} \(\Rightarrow \) \ref{main-t-example2-3}.
    If in addition \(E\) satisfies \ref{cnd:A7}, then \ref{main-t-example2-3} \(\Rightarrow \)\ref{main-t-example2-1}.
\end{theorem}
\begin{proof}
    The implication \ref{main-t-example2-2} \(\Rightarrow\) \ref{main-t-example2-3} follows from \Cref{main-omega-example}.
    The proofs of \ref{main-t-example2-1} \(\Rightarrow\) \ref{main-t-example2-2} and \ref{main-t-example2-3} \(\Rightarrow \) \ref{main-t-example2-1} (under the extra assumption \ref{cnd:A7}) are similar to those of the same implications in \Cref{main-t-example1} and are therefore left to the reader.
\end{proof}

\Cref{main-intro-applied} in the introduction now follows from \Cref{intro-example} and \Cref{main-t-example1,main-t-example2}.

\begin{remark}
    We believe the implications \((iii) \Rightarrow (i)\) in \Cref{main-t-example1,main-t-example2} hold without the additional assumption \ref{cnd:A7} but were unable to show this.
\end{remark}

\begin{remark} 
    Let \(\mathcal{V} = (v_j)_{j \in \N}\) be a weight function system on \(G\).
    \begin{enumerate}[i,wide=0pt]
        \item The condition \eqref{Q-inv} is always fulfilled if for all \(i \in \N\) there is \(j \geq i\) such that \(v_i/v_j\) vanishes at infinity. 
        \item The condition \eqref{Q-inv} means that the sequence \((1/v_n)_{n \in \N}\) is regularly decreasing in the sense of \cite{BMS}.
            We refer to \cite{BMS} for more information and various characterizations of \eqref{Q-inv}.
            A slight variant of the condition \eqref{Q-inv} is considered in \cite{Reiher}, where the quasinormability of certain weighted Fr\'{e}chet spaces of measurable functions is studied. 
        \item The condition \eqref{omegaw} is equivalent to
            \begin{equation}
                \label{omegaw1}
                \forall i \in \N \, \exists j \geq i \, \forall m \geq j \, \exists C >0, \theta \in (0,1) \, \forall g \in G \colon v^{\theta}_i(g) v^{1-\theta}_m(g) \leq C v_j(g).
            \end{equation}
            This follows by taking the infimum over \(\varepsilon\) in the right-hand side of the inequality in \eqref{omegaw}.
        \item Let \(G\) be connected and non-compact.
            Let \(d\) and \((w_j)_{j \in \N}\) be as in \Cref{example-df}.
            Set \((v_j)_{j \in \N} = (w_j \circ d)_{j \in \N}\).
            Then, \((v_j)_{j \in \N}\) satisfies \eqref{omegaw1} if and only if
            \[
                \forall i \in \N \, \exists j \geq i \, \forall m \geq j \, \exists C >0, \theta \in (0,1) \, \forall t \geq 0 \colon 
                w^{\theta}_i(t) w^{1-\theta}_m(t) \leq C w_j(t)
                .
            \]
            We refer to \cite[Section 6]{D-V} for various examples of sequences \((w_j)_{j \in \N}\) that (do or do not) satisfy this condition.
    \end{enumerate}
\end{remark}


\begin{thebibliography}{999}
    \setlength{\itemsep}{0pt}
    \bibitem{B-E}F.~Bastin, B.~Ernst, \emph{A criterion for \(CV(X)\) to be quasinormable}, Results Math. \textbf{14} (1988), 223--230.

    \bibitem{BK} J.~Berstein, B.~Kr\"{o}tz, \emph{Smooth Fr\'{e}chet globalizations of Harish-Chandra modules,} Israel J. Math. \textbf{199} (2014), 45--111.


    \bibitem{B-M} K.~D.~Bierstedt, R.~Meise, \emph{Distinguished echelon spaces and the projective
    description of weighted inductive limits of type \(V_dC(X)\)}, pp.\ 169--226, in \emph{Aspects of Mathematics and its Applications}, North-Holland Math. Library \textbf{34}, Amsterdam, 1986.



    \bibitem{BMS} K.~D.~Bierstedt, R.~Meise, W.~H.~Summers, \emph{A projective description of weighted inductive limits}, Trans. Amer. Math. Soc. \textbf{272} (1982), 107--160.



    \bibitem{B-D} J.~Bonet, S.~Dierolf, \emph{On the lifting of bounded sets in Fr\'{e}chet spaces}, Proc. Edinbourgh Math. Soc. \textbf{36} (1993), 277-281.


    \bibitem{BMT} R.~W.~Braun, R.~Meise, B.~A.~Taylor, \emph{Ultradifferentiable functions and Fourier analysis}, Results Math. \textbf{17} (1990), 206--237.
    \bibitem{C}

    P.~Cartier, \emph{Vecteurs diff\'{e}rentiables dans les repr\'{e}sentations unitaires des groupes de Lie,}
    Séminaire Bourbaki \textbf{17} (1976), exp. no. 454, 20--34.

    \bibitem{Debrouwere} A.~Debrouwere, \emph{Quasinormable \(C_0\)-groups and translation-invariant Fr\'{e}chet spaces of type \(\mathcal{D}_E\)}, Results Math. \textbf{74} (2019), 135.

    \bibitem{Debrouwere2}
    A.~Debrouwere, {\em Sequence space representations for spaces of entire functions with rapid decay on strips},
    J. Math. Anal. Appl. \textbf{497} (2021), 124872.


    \bibitem{D-K} A.~Debrouwere, T.~Kalmes, \emph{Linear topological invariants for kernels of convolution and differential operators}, J. Funct. Anal. \textbf{284} (2023), 109886.

    \bibitem{D-V} A.~Debrouwere, J.~Vindas, \emph{Topological properties of convolutor spaces via the short-time Fourier transform}, Trans. Amer. Math. Soc. \textbf{374} (2021), 829--861.

    \bibitem{D-P-V} P.~Dimovski, S.~Pilipovic, J.~Vindas, \emph{New distribution spaces associated to translation-invariant Banach spaces}, Monatsh. Math. \textbf{177} (2015), 495–515.

    \bibitem{D-M} J. Dixmier, P. Malliavin, \emph{Factorisations de fonctions et de vecteurs ind\'{e}finiment diff\'{e}rentiables}, Bull. des Sci. Math. \textbf{102} (1978), 305--330.

    \bibitem{F1987} H.~G.~Feichtinger, \emph{On a class of convolution algebras of functions,} Ann. Inst. Fourier \textbf{27}
    (1977), 135--162.

    \bibitem{F-G} H.~G.~Feichtinger, K.~Gröchenig, \emph{Banach spaces related to integrable group representations and their atomic decompositions. I}, J. Funct. Anal. \textbf{86} (1989), 307--340.

    \bibitem{FS} S.~Fürdös, G.~Schindl, \emph{The theorem of iterates for elliptic and non-elliptic operators}, J. Funct. Anal. \textbf{283} (2022), 109554.

    \bibitem{Garding47} L.~Gårding, \emph{Note on continuous representations of Lie groups}, Proc. Nat. Acad. Sci. U.S.A. \textbf{33} (1947), 331--332.

    \bibitem{Garding} L.~Gårding, \emph{Vecteurs analytiques dans les repr\'{e}sentations des groupes de Lie}, Bull. Soc. Math. France \textbf{88} (1960), 73--93.

    \bibitem{Gomez-Collado} M.~C.~Gómez-Collado, \emph{Almost periodic ultradistributions of Beurling and Roumieu type}, Proc. Amer. Math. Soc. \textbf{129} (2000), 2319--2329.

    \bibitem{Goodman1} R.~Goodman, \emph{Differential operators of infinite order on a Lie group I}, J. Math. Mech \textbf{19} (1970), 879--894.

    \bibitem{Goodman2} R.~Goodman, \emph{Differential operators of infinite order on a Lie group II}, Indiana Univ. Math. J. \textbf{21} (1970), 383--409.

    \bibitem{GW} R.~Goodmann, N.~R.~Wallach, \emph{Whittaker vectors and conical vectors}, J. Funct. Anal. \textbf{39} (1980), 199--279.

    \bibitem{Grothendieck} A.~Grothendieck, \emph{Sur les espaces \((F)\) et \((DF)\)}, Summa Brasil Math. \textbf{3} (1954), 57--123.

    \bibitem{Heinrich} T.~Heinrich, R.~Meise, \emph{A support theorem for quasianalytic functionals}, Math. Nachr. \textbf{280} (2007), 364--387.

    \bibitem{Hormander} L.~Hörmander, \emph{The analysis of linear partial differential operators. I. Distribution theory and Fourier analysis}, Springer-Verlag, Berlin, 1990.

    \bibitem{Komatsu}
    H.~Komatsu, {\em Ultradistributions I. Structure theorems and a characterization},
    J. Fac. Sci. Univ. Tokyo Sect. IA Math. 20 (1973), 25--105.

    \bibitem{Langenbruch2012} M.~Langenbruch, \emph{Bases in spaces of analytic germs}, Ann. Polon. Math. \textbf{106} (2012), 223--242.

    \bibitem{Langenbruch2016} M.~Langenbruch, \emph{On the diametral dimension of weighted spaces of analytic germs}, Studia Math. \textbf{233} (2016), 85--100.

    \bibitem{M-V} R.~Meise, D.~Vogt, \emph{Introduction to functional analysis}, Oxford University Press, New York, 1997.

    \bibitem{Nelson} E.~Nelson, \emph{Analytic vectors}, Ann. of Math. \textbf{70} (1959), 572--615.

    \bibitem{Poulsen} N.~S.~Poulsen, \emph{On \(C^\infty\)-vectors and intertwining bilinear forms for representations of Lie groups}, J. Funct. Anal. \textbf{9} (1972), 87--120.

    \bibitem{Reiher} K.~Reiher, \emph{Weighted inductive and projective limits of normed K\"{o}the function spaces}, Results. Math. \textbf{13} (1988), 147--161.

    \bibitem{Schwartz-vectorvaluedI}L. Schwartz, \emph{Th\'{e}orie des distributions à valeurs vectorielles. I,} Ann. Inst. Fourier (Grenoble) \textbf{7}
    (1957), 1--141.

    \bibitem{Schwartz} L.~Schwartz, \emph{Th\'{e}orie des distributions}, Hermann, Paris, 1966.

    \bibitem{Vogt-83} D.~Vogt, \emph{Sequence space representations of spaces of test functions and
    distributions}, pp.\ 405--443, in \emph{Functional analysis, holomorphy, and approximation
    theory}, Lecture Notes in Pure and Appl. Math. \textbf{83}, New York, 1983.

    \bibitem{VogtPDFG} D.~Vogt, \emph{On the solvability of \(P(D) f = g\) for vector valued functions}, RIMS Kokyoroku \textbf{508} (1983), 168--181.


    \bibitem{Vogt-87} D.~Vogt, \emph{On the functors \(\operatorname{Ext}^1(E;F)\) for Fr\'{e}chet spaces}, Studia Math. \textbf{85} (1987), 163--197.

    \bibitem{V-W} D.~Vogt, M.~Wagner, \emph{Charakterisierung der Quotientenräume von \(s\) und eine Vermutung von Martineau},
    Studia Math. \textbf{67} (1980), 225--240.

    \bibitem{Wengenroth} J.~Wengenroth, \emph{Derived functors in functional analysis}, Springer-Verlag, Berlin, 2003.

    \bibitem{Wolf} E.~Wolf, \emph{Weighted Fr\'{e}chet spaces of holomorphic functions}, Studia Math. \textbf{174} (2006), 255--275.


\end{thebibliography}
\end{document}